\documentclass[11pt]{article}
\usepackage{amsmath,epsfig,amssymb,amsbsy,verbatim,array,color,graphics,graphicx}
\usepackage{amssymb,amsmath,amsthm}
\usepackage{graphicx}%\usepackage[dvips]{graphics}
\usepackage{subdepth}

\newtheorem{thm}{Theorem}
\newtheorem{lemma}[thm]{Lemma}
\newtheorem{prop}[thm]{Proposition}

\newtheorem{clm}[thm]{Claim}
\newtheorem{prob}[thm]{Problem}

\definecolor{darkblue}{rgb}{0,0,0.7}
\definecolor{darkgreen}{rgb}{0,0.3,0}
\definecolor{darkred}{rgb}{0.7,0,0}

\begin{document}

% paper title
\title{List rainbow connection number of graphs\\[2ex]}
\author{Rongxia Tang$^\dag$ \and Henry Liu\thanks{Corresponding author}
\thanks{School of Mathematics, Sun Yat-sen University, Guangzhou 510275, China. E-mail addresses: {\tt tangrx3@mail2.sysu.edu.cn} (R.~Tang), {\tt liaozhx5@mail.sysu.edu.cn} (H.~Liu), {\tt shiyp9@mail2.sysu.edu.cn} (Y.~Shi), {\tt wangchm33@mail2.sysu.edu.cn} (C.~Wang)}
\and Yueping Shi$^\dag$ \and Chenming Wang$^\dag$
\\[2ex]
}
\date{12 March 2025}
% make the title area
\maketitle
\begin{abstract}
An edge-coloured path is \emph{rainbow} if all of its edges have distinct colours. Let $G$ be a connected graph. The \emph{rainbow connection number} of $G$, denoted by $rc(G)$, is the minimum number of colours in an edge-colouring of $G$ such that, any two vertices are connected by a rainbow path. The \emph{strong rainbow connection number} of $G$, denoted by $src(G)$, is the minimum number of colours in an edge-colouring of $G$ such that, any two vertices are connected by a rainbow geodesic (i.e., a path of shortest length). These two notions of connectivity of graphs were introduced by Chartrand, Johns, McKeon and Zhang in 2008. In this paper, we introduce the \emph{list rainbow connection number} $rc^\ell(G)$, and the \emph{list strong rainbow connection number} $src^\ell(G)$. These two parameters are the versions of $rc(G)$ and $src(G)$ that involve list edge-colourings. Among our results, we will determine the list rainbow connection number and list strong rainbow connection number of some specific graphs. %, including cycles, wheels, and complete bipartite and multipartite graphs. 
We will also characterise all pairs of positive integers $a$ and $b$ such that, there exists a connected graph $G$ with $src(G)=a$ and $src^\ell(G)=b$, and similarly for the pair $rc^\ell$ and $src^\ell$. Finally, we propose the question of whether or not we have $rc(G)=rc^\ell(G)$, for all connected graphs $G$.\\


\end{abstract}

\section{Introduction}

In this paper, all graphs are finite, simple, and undirected. Let $K_n$ and $C_n$ denote the complete graph and cycle on $n$ vertices. Let $K_{n_1,\dots,n_t}$ denote the complete $t$-partite graph with class sizes $n_1,\dots,n_t$. In particular, $K_{m,n}$ is the complete bipartite graph with class sizes $m$ and $n$. For $t_1,t_2,\ldots\ge 0$, let $K_{t_1\times 1,t_2\times 2,\dots}$ denote the complete multipartite graph with $t_i$ classes of size $i$, for all $i\ge 1$. Any subscript with $t_i=0$ is usually omitted. In particular, $K_{t\times n}$ denotes the complete $t$-partite graph where every class has size $n$. For a graph $G$, a vertex $v\in V(G)$, and a subset $S\subset V(G)$, let $\Gamma_G(v)=\{u\in V(G):uv\in E(G)\}$ and $\Gamma_G(S)=\bigcup_{w\in S}\Gamma_G(w)$ denote the \emph{open neighbourhood} of $v$ and $S$; and let $d_G(v)=|\Gamma_G(v)|$ denote the \emph{degree} of $v$. Let $G-v$ denote the subgraph obtained by deleting $v$ and all edges at $v$ from $G$. %Let $G+v$ denote the \emph{join} of $G$ and $v$, i.e., $G+v$ is obtained by connecting $v$ to all vertices of $G$. 
For a subset of edges $E\subset E(G)$, let $G-E$ denote the subgraph $(V(G),E(G)\setminus E)$. For two vertices $x,y\in V(G)$, let $d_G(x,y)$ denote the \emph{distance} from $x$ to $y$ in $G$, and diam$(G)=\max \{d_G(x,y):x,y\in V(G)\}$ denote the \emph{diameter} of $G$. For any other undefined terms in graph theory, we refer to the book by Bollob\'as \cite{Bol98}.

Throughout, let $r\ge 1$ be a positive integer. An \emph{$r$-colouring} (or simply \emph{colouring}) of a graph $G$ is a function $c:V(G)\to [r]$, where $[r]=\{1,\dots,r\}$, and $c$ is \emph{proper} if $c(u)\neq c(v)$ for all $uv\in E(G)$. We think of $[r]$ as a set of colours, and each vertex of $G$ is given one of $r$ possible colours. The \emph{chromatic number} $\chi(G)$ of $G$ is the minimum integer $r$ such that there exists a proper $r$-colouring of $G$. A \emph{list colouring} of $G$ is a generalisation of a colouring, where each vertex of $G$ has a prescribed list of colours, and may only be assigned a colour from its own list. The notion of list colouring was introduced in the late 1970s, independently by Vizing \cite{Viz76}, and by Erd\H{o}s, Rubin and Taylor \cite{ERT80}. More precisely, an \emph{$r$-list assignment} of $G$ is a function $L:V(G)\to \mathbb N^{(\ge r)}$, where $\mathbb N^{(\ge r)}$ denotes the family of all subsets of $\mathbb N$ with at least $r$ elements. For $v\in V(G)$, a set $L(v)$ is a \emph{list}. An \emph{$L$-colouring} of $G$ is a function $c:V(G)\to \mathbb N$ such that $c(v)\in L(v)$ for all $v\in V(G)$, and $c$ is \emph{proper} if $c(u)\neq c(v)$ for all  $uv\in E(G)$. We say that $G$ is \emph{$r$-choosable} if for every $r$-list assignment $L$ of $G$, there exists an $L$-colouring of $G$ which is proper. The \emph{list chromatic number} $\chi_\ell(G)$ of $G$ is the minimum integer $r$ such that $G$ is $r$-choosable. The parameter $\chi_\ell(G)$ may also be denoted by $ch(G)$, and may be called the \emph{choice number} of $G$. Clearly, we have $\chi(G)\le\chi_\ell(G)$. Equality does not necessarily hold. For example, we have $\chi(K_{m,n})=2$. A result of Erd\H{o}s, Rubin and Taylor \cite{ERT80} states that $\chi_\ell(K_{p,p})>r$ if $p={2r-1\choose r}$. Another well-known result from the 1970s states that $\chi_\ell(K_{p,p^p})=p+1$ for $p\ge 1$.

The edge-colouring formulations are very much analogous. An \emph{$r$-edge-colouring} (or simply \emph{edge-colouring} or \emph{colouring}) of $G$ is a function $c:E(G)\to [r]$, and $c$ is \emph{proper} if $c(e)\neq c(f)$ for any two incident edges $e,f\in E(G)$. The \emph{edge-chromatic number} (or \emph{chromatic index}) $\chi'(G)$ of $G$ is the minimum integer $r$ such that there exists a proper $r$-edge-colouring. An \emph{$r$-edge-list assignment} of $G$ is a function $L:E(G)\to \mathbb N^{(\ge r)}$. For $e\in E(G)$, a set $L(e)$ is a \emph{list}. An \emph{$L$-edge-colouring} of $G$ is a function $c:E(G)\to \mathbb N$ such that $c(e)\in L(e)$ for all $e\in E(G)$, and $c$ is \emph{proper} if $c(e)\neq c(f)$ for any two incident edges  $e,f\in E(G)$. The \emph{list edge-chromatic number} $\chi_\ell'(G)$ of $G$ is the minimum integer $r$ such that, for all $r$-edge-list assignment $L$ of $G$, there exists a proper $L$-edge-colouring. We have $\chi'(G)\le\chi_\ell'(G)$. While we may possibly have $\chi(G)<\chi_\ell(G)$ for some graph $G$, the famous \emph{list colouring conjecture}, which first appeared in print in a paper of Bollob\'as and Harris \cite{BH85}, states that $\chi'(G)=\chi_\ell'(G)$ for every graph $G$. 

The \emph{rainbow connection number}, introduced by Chartrand, Johns, McKeon and Zhang \cite{CJMZ08} in 2008, is a graph colouring parameter that has recently received significant interest. A path connecting two vertices $u,v$ of a connected graph $G$ is a \emph{geodesic} if it has the shortest possible length, i.e., $d_G(u,v)$. An edge-coloured path is \emph{rainbow} if its edges have distinct colours. An edge-colouring (not necessarily proper) of a connected graph $G$ is \emph{rainbow connected} (resp.~\emph{strongly rainbow connected}) if any two vertices of $G$ are connected by a rainbow path (resp.~rainbow geodesic). The \emph{rainbow connection number} $rc(G)$ and \emph{strong rainbow connection number} $src(G)$ are, respectively, the minimum integer $r$ such that there exists a rainbow connected/strongly rainbow connected colouring of $G$, using $r$ colours. Then clearly, we have $rc(G)\le src(G)$. 

%The two parameters may be arbitrarily far apart. Chern and Li \cite{CL12} proved that, for positive integers $a\le b$, there exists a connected graph $G$ such that $rc(G) = a$ and $src(G) = b$ if and only if $a = b\in\{1, 2\}$ or $3 \le a \le b$.

Here, we introduce list colouring versions of the rainbow connection number and strong rainbow connection number. Let $G$ be a connected graph. The \emph{list rainbow connection number} $rc^\ell(G)$ of $G$ is the minimum integer $r$ such that, for all $r$-edge-list assignment $L$ of $G$, there exists an $L$-edge-colouring $c$ of $G$ such that $c$ is rainbow connected. The \emph{list strong rainbow connection number} $src^\ell(G)$ of $G$ is the minimum integer $r$ such that, for all $r$-edge-list assignment $L$ of $G$, there exists an $L$-edge-colouring $c$ of $G$ such that $c$ is strongly rainbow connected. It is not difficult to see the following inequalities.
\begin{equation}\label{ineqs}
rc(G)\le rc^\ell(G),\quad src(G)\le src^\ell(G),\quad rc^\ell(G)\le src^\ell(G).
\end{equation}

This paper will be organised as follows. In Section \ref{gensect}, we shall prove some results about $rc^\ell(G)$ and $src^\ell(G)$ for general graphs $G$. In Section \ref{specsect}, we will determine $rc^\ell(G)$ and $src^\ell(G)$ for some specific graphs $G$, namely: cycles, wheels, complete bipartite and multipartite graphs, and the Petersen graph. In Section \ref{compsect}, we will characterise all pairs of integers $a\le b$ such that, there exists a graph with $src(G)=a$ and $src^\ell(G)=b$; and similarly for the pair $rc^\ell(G)$ and $src^\ell(G)$. In a similar direction to the list colouring conjecture, we propose the problem of whether of not we have $rc(G)=rc^\ell(G)$ for all connected graphs $G$, and we shall present some initial insights to this problem.

%Similar to the list colouring conjecture, we propose the following two conjectures.
%
%\begin{conj}[List rainbow connection number conjecture]\label{lrcconj}
%Let $G$ be a connected graph. Then
%\[
%rc(G)=rc^\ell(G).
%\]
%\end{conj}
%
%\begin{conj}[List strong rainbow connection number conjecture]\label{lsrcconj}
%Let $G$ be a connected graph. Then
%\[
%src(G)=src^\ell(G).
%\]
%\end{conj}

\section{List rainbow connection number of general graphs}\label{gensect}

In this section, we shall prove some results for the rainbow connection and strong rainbow connection parameters, for general graphs.

\begin{thm}\label{charthm}
 Let $G$ be a non-trivial connected graph on $n$ vertices.
\begin{enumerate}
\item[(a)] We have
\begin{eqnarray}
& \textup{diam}(G)\le rc(G)\le rc^\ell(G)\le n-1, &\label{ineqs2}\\
& \textup{diam}(G)\le rc(G)\le src(G)\le src^\ell(G)\le e(G). & \label{ineqs3}
\end{eqnarray}
\item[(b)] Suppose that $G$ has $p$ bridges. Then $rc(G)\ge p$.
\item[(c)] Let $v$ be a cut-vertex of $G$ such that $G-v$ has $q\ge 2$ components. Then $src(G)\ge q$.
\item[(d)] Let $n\ge 3$, and $w$ be a pendent vertex of $G$. Then $rc(G)\ge rc(G-w)$ and $src(G)\ge src(G-w)$.
\item[(e)] The following are equivalent.
\begin{enumerate}
\item[(i)] $G$ is a complete graph.
\item[(ii)] \textup{diam}$(G)=1$.
\item[(iii)] $rc(G)=1$.
\item[(iv)] $src(G)=1$.
\item[(v)] $rc^\ell(G)=1$.
\item[(vi)] $src^\ell(G)=1$.
\end{enumerate}
\item[(f)] We have $rc(G)=2$ if and only if $src(G)=2$. Moreover, $src^\ell(G)=2$ implies $rc^\ell(G)=2$, which in turn implies both $rc(G)=2$ and $src(G)=2$.
%\item[(b)] $rc^\ell(G)=2$ if and only if $src^\ell(G)=2$. ***Probably false***
%\item[(c)] The following are equivalent.
%\begin{enumerate}
%\item[(i)] $G$ is a tree.
%\item[(ii)] $rc(G)=e(G)$.
%\item[(iii)] $src(G)=e(G)$.
%\item[(iii)] $rc^\ell(G)=e(G)$.
%\item[(v)] $src^\ell(G)=e(G)$.
%\end{enumerate}
%Moreover, if any of (i), (ii) or (iii) holds, then $src(G)=e(G)$, which in turn implies that $src^\ell(G)=e(G)$.
\item[(g)] The following are equivalent.
\begin{enumerate}
\item[(i)] $G$ is a tree.
\item[(ii)] $rc(G)=e(G)$.
\item[(iii)] $src(G)=e(G)$.
\item[(iv)] $rc^\ell(G)=e(G)$.
\item[(v)] $src^\ell(G)=e(G)$.
\end{enumerate}
\end{enumerate}
\end{thm}

\begin{proof}
(a) We verify the last inequality of (\ref{ineqs2}). Let $L$ be an $(n-1)$-edge-list assignment of $G$. Let $T\subset G$ be a spanning tree of $G$. We obtain an $L$-edge-colouring $c$ of $G$ where the colours $c(e)$ are distinct over $e\in E(T)$, and the colours of the edges of $E(G)\setminus E(T)$ are chosen arbitrarily. Then $c$ is rainbow connected, and thus $rc^\ell(G)\le n-1$. %The last inequality of (\ref{ineqs3}) may be obtained similarly, by considering an $e(G)$-edge-list assignment $L$ of $G$, and similarly constructing an $L$-edge-colouring $c$. We have $c$ is strongly rainbow connected, and thus $src^\ell(G)\le e(G)$. 
The remaining inequalities of (\ref{ineqs2}) and (\ref{ineqs3}) are clear.\\[1ex]
\indent(b) Let $c$ be an edge-colouring of $G$, using fewer than $p$ colours. Then there exist bridges $e,f\in E(G)$ such that $c(e)=c(f)$. It is not hard to see that, one end-vertex of $e$ is not connected to one end-vertex of $f$ by a rainbow path. 
Hence, $rc(G)\ge p$.\\[1ex]
\indent(c) Let $c$ be an edge-colouring of $G$, using fewer than $q$ colours. Then there exist $u,w\in V(G-v)$ such that $uv,wv\in E(G)$, $c(uv)=c(wv)$, and $u,w$ are in different components of $G-v$. There is no rainbow geodesic connecting $u$ and $w$. Hence, $src(G)\ge q$.\\[1ex]
\indent(d) Let $G$ be given an edge-colouring with fewer than $rc(G-w)$ colours. Since $|V(G-w)|\ge 2$, there exist $x,y\in V(G-w)$ which are not connected by a rainbow path in $G-w$, and also in $G$. Hence, $rc(G)\ge rc(G-w)$. The inequality $src(G)\ge src(G-w)$ is proved similarly.\\[1ex]
\indent(e) In \cite{CLLL18}, it was proved that (i) to (iv) are equivalent. Clearly we have (i) $\Rightarrow$ (vi), and by (\ref{ineqs}), we have $rc(G)\le rc^\ell(G)\le src^\ell(G)$, and hence (vi) $\Rightarrow$ (v) $\Rightarrow$ (iii).\\[1ex]
%\indent(b) Suppose that $rc^\ell(G)=2$. Then (a) implies that diam$(G)\ge 2$, so that $2\le\textup{diam}(G)\le rc(G)\le rc^\ell(G)=2$, and $rc(G)=2$. Conversely, suppose that $rc^\ell(G)=2$. ***Maybe false***\\[1ex]
%\indent(c) In \cite{CJMZ08}, it was proved that (i) and (ii) are equivalent. If (ii) holds, then by (\ref{ineqs2}), we have $e(G)=rc(G)\le rc^\ell(G)\le n-1\le e(G)$, so that (iii) holds. Suppose that (iii) holds. If $G$ is not a tree, then by (\ref{ineqs2}), we have $e(G)>n-1\ge rc^\ell(G)$, a contradiction. Thus (i) holds. For the final part, suppose that (ii) holds. Then by (\ref{ineqs3}),
%\[
%e(G)=rc(G)\le src(G)\le e(G),
%\]
%so that $src(G)=e(G)$. Also, if $src(G)=e(G)$, then again by (\ref{ineqs3}),
%\[
%e(G)=src(G)\le src^\ell(G)\le e(G),
%\]
%so that $src^\ell(G)=e(G)$. \\[1ex]
\indent(f) The equivalence $rc(G)=2\Longleftrightarrow src(G)=2$ was proved in \cite{CJMZ08}. The implications $src^\ell(G)=2\Rightarrow rc^\ell(G)=2\Rightarrow rc(G)=2$ follow from the inequalities $src^\ell(G)\ge rc^\ell(G)\ge rc(G)$ and (e).\\[1ex]
%
%(i) This was proved in \cite{CJMZ08}. We provide a proof for the sake of completeness. If $src(G)=2$, then we have $rc(G)\le src(G)=2$, and thus $rc(G)=2$ by (d). Conversely, suppose that $rc(G)=2$. Then, a rainbow connected $2$-colouring of $G$ is also strongly rainbow connected, since any two non-adjacent vertices of $G$ are connected by a rainbow path of length $2$, which is also a geodesic. Hence, $2=rc(G)\le src(G)\le 2$, so that $src(G)=2$. \\
%\indent(e)(ii) This is proved similarly as (i). If $src^\ell(G)=2$, then we have $rc^\ell(G)\le src^\ell(G)=2$ by (\ref{ineqs}), and thus $rc^\ell(G)=2$ by (d). Conversely, suppose that $rc^\ell(G)=2$. Then by (\ref{ineqs2}), we have diam$(G)\le rc^\ell(G)=2$, and in fact, diam$(G)=2$ by (d). Now for any $2$-edge-list assignment $L$ of $G$, there exists an $L$-edge-colouring $c$ which is rainbow connected. Thus under $c$, any two non-adjacent vertices of $G$ are connected by a rainbow path of length $2$, which is also a geodesic.  Hence, $2=rc^\ell(G)\le src^\ell(G)\le 2$, so that $src^\ell(G)=2$. \\
%\indent To see the ``moreover'' part, it suffices to show that $rc^\ell(G)=2$ implies $rc(G)=2$. If $rc^\ell(G)=2$, then by (\ref{ineqs2}), we have $rc(G)\le rc^\ell(G)=2$, and thus $rc(G)=2$ by (d). \\[1ex]
\indent(g) In \cite{CJMZ08}, it was proved that (i) and (ii) are equivalent. Now, since $rc(G)\le src(G)\le src^\ell(G)\le e(G)$ by (\ref{ineqs3}), we have (ii) $\Rightarrow$ (iii) $\Rightarrow$ (v). Also, since $rc(G)\le rc^\ell(G)\le src^\ell(G)\le e(G)$ by (\ref{ineqs}) and (\ref{ineqs3}), we have (ii) $\Rightarrow$ (iv) $\Rightarrow$ (v).\\
\indent Hence, it suffices to prove that (v) $\Rightarrow$ (i). Suppose that $G$ is not a tree. Let $L$ be an $(e(G)-1)$-edge-list assignment of $G$. We construct an $L$-edge-colouring of $G$ which is strongly rainbow connected, so that $src^\ell(G)<e(G)$. We may assume $|L(e)|=e(G)-1$ for all $e\in E(G)$.\\
\indent Suppose first that $L$ is not constant on $E(G)$. Consider the bipartite graph $H$ with classes $X=E(G)$ and $Y=\bigcup_{e\in E(G)}L(e)$, and connect $e\in X$ to $\alpha\in Y$ with an edge if and only if $\alpha\in L(e)$. We easily have Hall's condition for $H$ that $|\Gamma_H(S)|\ge |S|$ for all $S\subset X$. By Hall's theorem \cite{Hal35}, $H$ has a matching of size $|X|=e(G)$. Hence, there exists an $L$-edge-colouring of $G$ where all edges have distinct colours, and thus is strongly rainbow connected.

%We verify Hall's condition for $H$ that $|\Gamma_H(S)|\ge |S|$ for all $S\subset X$. If $|S|\le e(G)-1$, then $|\Gamma_H(S)|\ge |S|$ holds since $d_H(e)= e(G)-1$ for all $e\in X$. If $|S|=e(G)$, then since $L$ is not constant on $E(G)$, we have $|\Gamma_H(S)|\ge e(G) =|S|$. Thus, by Hall's theorem \cite{Hal35}, $H$ has a matching of size $|X|=e(G)$. This means that there exists an $L$-edge-colouring of $G$ where all edges have distinct colours, and this $L$-edge-colouring is clearly strongly rainbow connected.

\indent Now suppose that $L$ is constant on $E(G)$, say $L\equiv\{1,2,\dots,e(G)-1\}$ on $E(G)$. Since $G$ is not a tree, we may choose a cycle $C=v_0v_1\cdots v_{k-1}v_0\subset G$ with minimum length $k\ge 3$. Let $c$ be an $L$-edge-colouring of $G$ where $c(v_0v_1)=c(v_{\lfloor k/2\rfloor}v_{\lfloor k/2\rfloor +1})=1$, and the other edges of $G$ have distinct colours from $\{2,3,\dots,e(G)-1\}$. Let $x,y\in V(G)$, and $P=u_0u_1\cdots u_\ell$ be an $x-y$ geodesic, where $u_0=x$ and $u_\ell=y$. If $P$ is not rainbow coloured, then we have $c(u_iu_{i+1})=c(u_ju_{j+1})=1$ for some $0\le i<j<\ell$. We may assume $\{u_i,u_{i+1}\}=\{v_{\lfloor k/2\rfloor},v_{\lfloor k/2\rfloor +1}\}$ and $\{u_j,u_{j+1}\}=\{v_0,v_1\}$. Consider moving along $C$ from $u_{j+1}$ in the direction away from $u_j$. Let $j+1\le r\le \ell$ be the maximum integer such that $u_{j+1}\cdots u_r\subset C$. Continuing along $C$ from $u_r$, since we must be able to reach $v_{\lfloor k/2\rfloor}$ or $v_{\lfloor k/2\rfloor +1}$, we must be able to return to $P$ for the first time, say at $u_s\in V(P)\setminus\{u_j,u_{j+1},\dots,u_r\}$. Now, the path $Q$ from $u_r$ to $u_s$ along $C$ has $e(Q)\le \lceil\frac{k-2}{2}\rceil$. Since $P$ is an $x-y$ geodesic, we have $e(Q)\ge 2$, and the path $R$ from $u_r$ to $u_s$ along $P$ has $1\le e(R)\le e(Q)\le \lceil\frac{k-2}{2}\rceil$. But now, $Q\cup R$ is a cycle in $G$ with $e(Q\cup R)\le 2\lceil\frac{k-2}{2}\rceil<k$, which contradicts the choice of $C$. We conclude that $P$ is rainbow coloured, and $c$ is a strongly rainbow connected $L$-edge-colouring of $G$.
\end{proof}

We remark that in Theorem \ref{charthm}(f), $src(G)=2$ does not imply $src^\ell(G)=2$, and $rc^\ell(G)=2$ does not imply $src^\ell(G)=2$. However, we are not certain if $rc(G)=2$ implies $rc^\ell(G)=2$. We shall discuss this in more detail in Section \ref{compsect}, where will will see that, given $2\le a\le b$, there exist connected graphs $G,G'$ such that $src(G)=a$ and $src^\ell(G)=b$; and $rc^\ell(G')=a$ and $src^\ell(G')=b$.

Next, we observe that if $G$ is a connected graph, and $H$ is a connected spanning subgraph of $G$, then $rc^\ell(G)\le rc^\ell(H)$. However, this fact does not extend to the parameter $src^\ell$.

\begin{thm}\label{spanthm}
There exist connected graphs $G$ and $H$ such that, $H$ is a spanning subgraph of $G$, and $src^\ell(G)> src^\ell(H)$.
\end{thm}

\begin{proof}
We consider an example of Chen et al.~\cite{CLLL18}. In Figure 1(a), let $H$ be the graph consisting of the solid edges, and $G$ be the graph obtained from $H$ by adding the dashed edge. Let $e_1,\dots,e_{12}$ be the edges of $H$ as shown. In \cite{CLLL18}, it was shown that $src(H)=4$ and $src(G)\ge 5$. A strongly rainbow connected, $4$-edge-colouring of $H$ is shown in Figure 1(b).
\indent\\[-0.25cm]
\begin{figure}[htp]
\centering
\includegraphics[width=13.5cm]{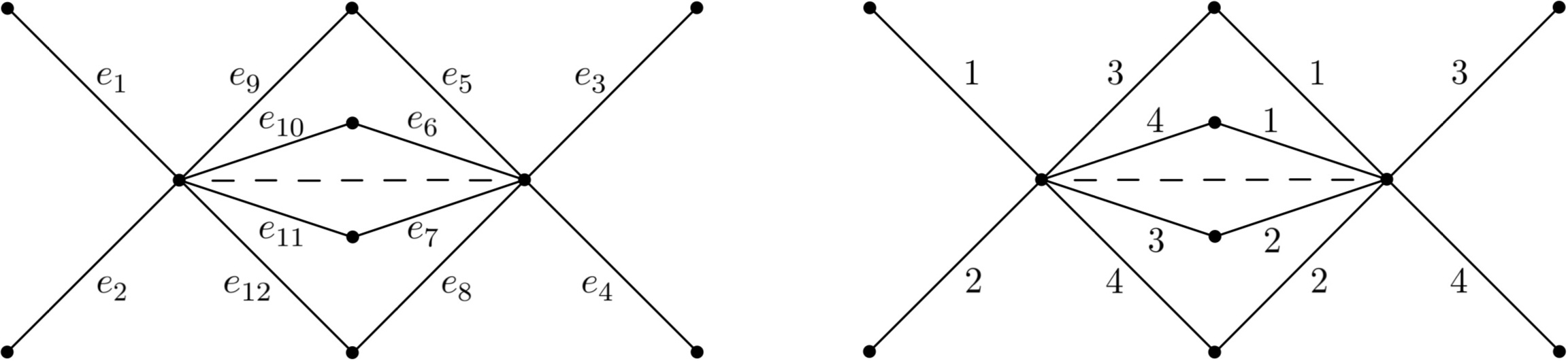}
\end{figure}
\indent\\[-0.9cm]
\begin{center}
\hspace{0.025cm}(a)\hspace{6.9cm}(b)\\[2ex]
\hspace{0.02cm}Figure 1. The graphs $G$ and $H$
\end{center}
\indent\\[-0.7cm]
%\indent\\[0cm]
%$e_1\quad e_2\quad e_3\quad e_4\quad e_5\quad e_6\quad e_7\quad e_8\quad e_9\quad e_{10}\quad e_{11}\quad e_{12}\quad 1\quad 2\quad 3\quad 4$\quad (a)\quad (b)\\[2cm]

\indent By (\ref{ineqs}), we have $src^\ell(G)\ge src(G)\ge 5$. We will prove that $src^\ell(H)\le 4$ (and thus $src^\ell(H)=4$ since we also have $src^\ell(H)\ge src(H)=4$). We note that the graph $H$ as drawn in Figure 1(a) is symmetric about the central vertical line. In the proof below, by the term \emph{symmetry}, we mean that we will make use of this symmetry, where the roles of the index sets $\{1,2\}$ and $\{3,4\}$ are interchanged, and similarly for $\{5,6,7,8\}$ and $\{9, 10, 11, 12\}$.

Let $L$ be a $4$-edge-list assignment of $H$. We may assume $|L(e_i)|=4$ for all $1\le i\le 12$. We construct a strongly rainbow connected $L$-edge-colouring $c$ of $H$. First, choose distinct $c(e_i)\in L(e_i)$ for $1\le i\le 4$. We may assume for now that $c(e_i)=i$ for $1\le i\le 4$, and update these later if necessary. Let $L_i'=L(e_i)\setminus\{3,4\}$ for $5\le i\le 8$, and $L_i'=L(e_i)\setminus\{1,2\}$ for $9\le i\le 12$. For $5\le i\le 12$, we have $2\le |L_i'|\le 4$. We will choose $c(e_i)\in L_i'$, where $c(e_i)$ is chosen arbitrarily if no choice is defined. %We have the following simple observation.
%\begin{itemize}
%\item[($\ast$)] Let $9\le j,k\le 12$. Suppose that $c$ satisfies $c(e_i)=i$ for $1\le i\le 4$, $c(e_j)=c(e_k)=a$, for some $a\neq 3$, and $c(e_{j-4}), c(e_{k-4}), a$ are distinct. Then under $c$, there exist rainbow geodesics from $u_1$ and $u_2$ to $u_3$. Similar statements are obtained by replacing $3$ with $4$, and by applying symmetry.
%\end{itemize}
We first have the following claim.
\begin{clm}\label{spanthmclm1}
Suppose that for some $\alpha\in\{1,2\}$, $\beta\in\{3,4\}$, and $\gamma_5,\dots,\gamma_8\not\in\{1,2,3,4\}$, where the $\gamma_j$ are not necessarily distinct, we have $L_i'=\{\alpha,\gamma_i\}$ for $5\le i\le 8$, and $L_i'=\{\beta,\gamma_{i-4}\}$ for $9\le i\le 12$. Then there exists a strongly rainbow connected $L$-edge-colouring $c$ of $H$.
\end{clm}

\begin{proof}
We construct an $L$-edge-colouring $c$ as follows. Assume that $\alpha=1$ and $\beta=3$, otherwise we may switch $1$ and $2$; or $3$ and $4$. Then $L(e_i)=\{3,4,1,\gamma_i\}$ for $5\le i\le 8$, and $L(e_i)=\{1,2,3,\gamma_{i-4}\}$ for $9\le i\le 12$. Suppose that there exists $\delta_3\in L(e_3)\setminus\{1,2,3,4\}$, say $\delta_3=5$. If $\gamma_5\neq 5$, then we update $c(e_3)=5$, and let $c(e_5)=\gamma_5$, $c(e_i)=3$ for $i=6,9,11$; $c(e_i)=1$ for $i=7,8$; and $c(e_i)=\gamma_{i-4}$ for $i=10,12$. Otherwise, it follows that $\gamma_i=5$ for $5\le i\le 8$. If there exists $\delta_2\in L(e_2)\setminus\{1,2,4,5\}$, we update $c(e_2)=\delta_2$, $c(e_3)=5$, and let $c(e_i)=1$ for $i=5,6$; $c(e_i)=3$ for $i=7,8$; $c(e_i)=2$ for $i=9,11$; and $c(e_i)=5$ for $i=10,12$. Otherwise, we have $L(e_2)=\{1,2,4,5\}$. If there exists $\delta_1\in L(e_1)\setminus\{1,2,4,5\}$, then we update $c(e_2)=1$, and repeat the previous argument to obtain $L(e_1)=\{1,2,4,5\}$. By symmetry, we may then obtain $L(e_4)=L(e_3)=\{3,2,4,5\}$. We let $c(e_i)=1$ for $i=1,5,6$; $c(e_i)=2$ for $i=4,9,11$; $c(e_i)=3$ for $i=3,10,12$; and $c(e_i)=4$ for $i=2,7,8$.

Hence, we have $L(e_3)=\{1,2,3,4\}$. By symmetry, we have $L(e_1)=\{1,2,3,4\}$. Recall that $L(e_i)=\{3,4,1,\gamma_i\}$ for $5\le i\le 8$, and $L(e_i)=\{1,2,3,\gamma_{i-4}\}$ for $9\le i\le 12$. If there exists $\delta_2\in L(e_2)\setminus\{1,2,3,4\}$, say $\delta_2=5$, we update $c(e_2)=5$, and let $c(e_i)=\gamma_i$ for $i=5,6$; $c(e_i)=1$ for $i=7,8$; $c(e_i)=2$ for $i=9,11$; and $c(e_i)=3$ for $i=10,12$. Otherwise, we have $L(e_2)=\{1,2,3,4\}$, and by symmetry, we have $L(e_4)=\{1,2,3,4\}$. We let $c(e_i)=1$ for $i=3,9,11$; $c(e_i)=2$ for $i=4,10,12$; $c(e_i)=3$ for $i=1,5,6$; and $c(e_i)=4$ for $i=2,7,8$.

In all cases, it is easy to verify that $c$ is a strongly rainbow connected colouring.
\end{proof}
Now, suppose $\big|\bigcup_{i=5}^8 L_i'\big|=\big|\bigcup_{i=9}^{12} L_i'\big|=2$. We construct an $L$-edge-colouring $c$ as follows. We have $L_i'=\{\alpha,\beta\}$ for $5\le i\le 8$, and $L_i'=\{\gamma,\delta\}$ for $9\le i\le 12$, for some $\alpha,\beta\not\in\{3,4\}$ and $\gamma,\delta\not\in\{1,2\}$. If $\alpha,\beta,\gamma,\delta$ are distinct, let $c(e_i)=\alpha$ for $i=5,6$; $c(e_i)=\beta$ for $i=7,8$; $c(e_i)=\gamma$ for $i=9,11$; and $c(e_i)=\delta$ for $i=10,12$. Now, assume that $\alpha=\gamma=5$. If $\beta\in\{1,2\}$ and $\delta\in\{3,4\}$, then we may apply Claim \ref{spanthmclm1}. Otherwise, assume that $\beta=6$. We let $c(e_i)=5$ for $i=5,6,9,11$; $c(e_i)=6$ for $i=7,8$; and $c(e_i)=\delta$ for $i=10,12$. In all cases, we have a strongly rainbow connected $L$-edge-colouring.

Hence, assume without loss of generality that $\big|\bigcup_{i=5}^8 L_i'\big|\ge 3$. Then there exist $\alpha_i\in L_i'$ such that at most two of the $\alpha_i$ are equal, for $5\le i\le 8$. Let $I=\{5\le i\le 8: \alpha_i\not\in\{1,2\}\}$. We have the following claim.

\begin{clm}\label{spanthmclm2}
If one of the following holds, then there exists a strongly rainbow connected $L$-edge-colouring of $H$.
\begin{enumerate}
\item[(i)] For some $5\le j\neq k\le 8$, we have $1\in L_j'$ and $2\in L_k'$.
\item[(ii)] We do not have $L_{i+4}'=\{\beta,\alpha_i\}$ over all $i\in I$, for some $\beta\in\{3,4\}$.
\item[(iii)] $I=\{5,6,7,8\}$.
\end{enumerate}

A similar result holds by applying the symmetry, if $\big|\bigcup_{i=9}^{12} L_i'\big|\ge 3$. Let (i$'$), (ii$'$), (iii$'$) denote the corresponding statements.
\end{clm}

\begin{proof}
Suppose that (i) holds. Assume that we may choose $c(e_5)=1\in L_5'$ and $c(e_6)=2\in L_6'$. We may assume that we can choose $c(e_7)=\gamma_7\in L_7'\setminus\{1,2\}$, otherwise we have $L_7'=L_8'=\{1,2\}$, and we may instead let $c(e_7)=1\in L_7'$,  $c(e_8)=2\in L_8'$, and either $c(e_5)\in L_5'\setminus\{1,2\}$ or $c(e_6)\in L_6'\setminus\{1,2\}$. We then let $c(e_8)=\gamma_8\in L_8'\setminus\{\gamma_7\}$, and we may assume $\gamma_8\neq 2$, but $\gamma_8=1$ is possible. Next, let $c(e_{11})\in L_{11}'\setminus\{\gamma_7\}$, and we may assume $c(e_{11})\neq 4$ (otherwise, switch $3$ and $4$). Finally, let $c(e_9)\in L_9'\setminus\{3\}$, $c(e_{10})\in L_{10}'\setminus\{3\}$, and $c(e_{12})\in L_{12}'\setminus\{c(e_9)\}$. Then, $c$ is a strongly rainbow connected $L$-edge-colouring.

Next, suppose that (ii) holds. Let $c(e_i)=\alpha_i$ for $5\le i\le 8$. If we can choose $c(e_{j+4})\in L_{j+4}'\setminus\{3,4,\alpha_j\}$ for some $j\in I$, let $c$ be such that $c(e_{p+4})\in L_{p+4}'$ and $c(e_{q+4})\in L_{q+4}'$ are distinct, if $\alpha_p=\alpha_q$. Then, $c$ is strongly rainbow connected. Otherwise, we have $L_{i+4}'\subset\{3,4,\alpha_i\}$ for all $i\in I$. Since there are at least three distinct colours among the $\alpha_i$, if (i) does not hold, then $|I|\ge 2$. Thus, $3\in L_{j+4}'$ and $4\in L_{k+4}'$ for some $j,k\in I$ with $j\neq k$. Using (i$'$), we have a strongly rainbow connected  $L$-edge-colouring.

Finally, suppose that (iii) holds. If (ii) and (ii$'$) do not hold, then $L_{i+4}'=\{\beta,\alpha_i\}$ over all $i\in I$, for some $\beta\in\{3,4\}$; and $L_i'=\{\beta',\alpha_i\}$ over all $i\in I$, for some $\beta'\in\{1,2\}$. By Claim \ref{spanthmclm1}, we have a strongly rainbow connected  $L$-edge-colouring.
\end{proof}

Now, note that the condition (i) in Claim \ref{spanthmclm2} does not depend on the $\alpha_i$. Suppose that (i) does not hold. We may assume that $\alpha_5,\dots,\alpha_8$ satisfy the following. If they are distinct, we assume $\alpha_6,\alpha_7,\alpha_8\not\in\{1,2\}$. Otherwise, we assume that $\alpha_7=\alpha_8$ is the only equality. We assume $\alpha_7,\alpha_8\not\in\{1,2\}$, otherwise we may change $\alpha_7$ to some $\alpha_7'\in L_7'\setminus\{\alpha_7\}$. We may further assume $\alpha_6\not\in\{1,2\}$. We shall apply Claim \ref{spanthmclm2} for such $\alpha_5,\dots,\alpha_8$.

If $\alpha_5\not\in\{1,2\}$, then $I=\{5,6,7,8\}$, and Claim \ref{spanthmclm2}(iii) holds. 

Now, let $\alpha_5\in\{1,2\}$, so $I=\{6,7,8\}$. Assume $\alpha_5=1$. If Claim \ref{spanthmclm2}(ii) does not hold, we may assume $L_{i+4}'=\{3,\alpha_i\}$ for $i\in I$. If $4\in L_9'$, then since $3\in L_{10}'$, Claim \ref{spanthmclm2}(i$'$) holds. Next, suppose $3\in L_9'$. If Claim \ref{spanthmclm2}(ii$'$) does not hold, then $L_i'=\{1,\alpha_i\}$ for all $i\in I$. If we can choose $c(e_5)=\beta_5\in L_5'\setminus\{1,2\}$ and $c(e_9)=\beta_9\in L_9'\setminus\{3,4\}$ such that $\beta_5\neq \beta_9$, then let $c(e_6)=\alpha_6$, $c(e_i)=1$ for $i=7,8$; $c(e_i)=3$ for $i=10,12$; and $c(e_{11})=\alpha_7$, so that $c$ is strongly rainbow connected. Otherwise, $L_5'=\{1,\beta\}$ and $L_9'=\{3,\beta\}$ for some $\beta\not\in\{1,2,3,4\}$. By Claim \ref{spanthmclm1}, there exists a strongly rainbow connected $L$-edge-colouring. 

Finally, suppose $3,4\not\in L_9'$. If there exists $\alpha_9\in L_9'$ such that, among $\alpha_6,\alpha_7,\alpha_8,\alpha_9$, at most two are equal, then Claim \ref{spanthmclm2}(iii$'$) holds by taking $\alpha_i\in L_{i+4}'$ for $5\le i\le 8$. Otherwise, we have $\alpha_7=\alpha_8$ and $L_9'=\{\alpha_6,\alpha_7\}$. If Claim \ref{spanthmclm2}(ii$'$) does not hold for $\alpha_6\in L_9'$, $3\in L_{10}'$ and $\alpha_7\in L_{11}',L_{12}'$, then $L_5'=\{1,\alpha_6\}$ and $L_7'=L_8'=\{1,\alpha_7\}$. Let $c(e_i)=\alpha_6\in L_i'$ for $i=5,6$; $c(e_i)=1\in L_i'$ for $i=7,8$; $c(e_i)=\alpha_7\in L_i'$ for $i=9,11$; and $c(e_i)=3\in L_i'$ for $i=10,12$. Then, $c$ is strongly rainbow connected.

\indent In all cases, we have obtained a strongly rainbow connected $L$-edge-colouring of $H$. This completes the proof of Theorem \ref{spanthm}.
\end{proof}

To conclude this section, we have the following result about graphs with a universal vertex. This result will be useful for the rest of the paper.

\begin{thm}\label{univthm}
Let $G$ be a non-complete graph with a universal vertex $v$. Suppose that $G-v$ has $q$ components $H_1,\dots,H_q$, where $p$ of the components are trivial.
\begin{enumerate}
\item[(a)] For $q\ge 3$, we have $rc(G)=rc^\ell(G)=\max(p,3)$.
\item[(b)] For $q=1,2$, we have $2\le rc(G)\le rc^\ell(G)\le 3$. 
\end{enumerate}
\end{thm}

\begin{proof}
Assume that $H_1,\dots,H_p$ are precisely the trivial components of $G-v$, say $V(H_i)=\{w_i\}$ for $1\le i\le p$. Let $r=\max(p,3)$, and $L$ be an $r$-edge-list assignment of $G$. We obtain an $L$-edge-colouring $f$ of $G$. Let $f(w_iv)\in L(w_iv)$ be distinct for $1\le i\le p$. For $p<i\le q$, choose a spanning tree $T\subset H_i$, and an ordering $e_1,\dots,e_t$ of the edges of $T$ such that for all $1\le k\le t$, the graph induced by $e_1,\dots,e_k$ is a subtree of $T$. Let $e_1=u_0u_1$, and for $2\le k\le t$, let $e_k=u_{\ell_k}u_k$, where $0\le \ell_k<k$, and $u_k$ is not a vertex of any of $e_1,\dots,e_{k-1}$. Now, let $f(u_0v)\in L(u_0v)$, $f(u_1v)\in L(u_1v)$, $f(e_1)\in L(e_1)$ be distinct. For $2\le k\le t$, suppose that $f(u_hv)\in L(u_hv)$ and $f(e_h)\in L(e_h)$ have been chosen such that $f(u_{\ell_h}v)$, $f(u_hv)$, $f(e_h)$ are distinct, for all $1\le h< k$. Let $f(u_kv)\in L(u_kv)\setminus\{f(u_{\ell_k}v)\}$ and $f(e_k)\in L(e_k)\setminus\{f(u_{\ell_k}v)\}$ be distinct. Repeat inductively, and apply this procedure for all such $H_i$. For all remaining edges of $H$, choose arbitrary colours from their lists.

Now, for non-adjacent $x\in V(H_i)$ and $y\in V(H_j)$, where $1\le i\le j\le q$, we have either $xvy$ or $xvzy$ is a rainbow path, for some $z\in V(H_j)\setminus\{x,y\}$. Hence, $f$ is a rainbow connected colouring of $G$, and $rc^\ell(G)\le r=\max(p,3)$.

By Theorem \ref{charthm}(e) and (\ref{ineqs}), if $q=1,2$, then the inequalities $2\le rc(G)\le rc^\ell(G)\le 3$ of (b) follow. Now, let $q\ge 3$. Since the bridges of $G$ are precisely $w_1v,\dots,w_pv$, we have $rc(G)\ge p$ by Theorem \ref{charthm}(b). Also, suppose that $g$ is a $2$-edge-colouring of $G$. Then, there exist $x\in V(H_i)$ and $y\in V(H_j)$, for some $1\le i<j\le q$, such that $g(xv)=g(yv)$. There is no rainbow $x-y$ path in $G$ under $g$. Thus, $rc(G)\ge 3$, and $rc(G)\ge \max(p,3)$. Hence,  $rc(G)=rc^\ell(G)=\max(p,3)$ by (\ref{ineqs}), and (a) holds.
\end{proof}

\section{List rainbow connection number of specific graphs}\label{specsect}

In \cite{CJMZ08}, Chartrand et al.~proved many results about the parameters $rc(G)$ and $src(G)$ when $G$ is some specific graph. In this section, we shall prove analogous results for the parameters $rc^\ell(G)$ and $src^\ell(G)$.

We first gather some known results. The list chromatic numbers of the even cycle $C_{2p}$ and the complete $t$-partite graph $K_{t\times 2}$ were determined by Erd\H{o}s, Rubin and Taylor \cite{ERT80}.

\begin{thm}\label{ERTthm}\textup{\cite{ERT80}}
\begin{enumerate}
\item[(a)] For $p \ge 2$, we have $\chi_\ell(C_{2p}) = 2$.
\item[(b)] For $t \ge 2$, we have $\chi_\ell(K_{t\times 2}) = t$.
\end{enumerate}
\end{thm}

The list chromatic number of $K_{t_1\times 1,t_3\times 3}$ was determined by Ohba \cite{Ohb04}. The case for $K_{t\times 3}$ was determined by Kierstead \cite{Kie00}.

\begin{thm}\label{KieOhbthm}\textup{\cite{Kie00,Ohb04}}
We have $\chi_\ell(K_{t_1\times 1,t_3\times 3})=\max(t,\lceil\frac{n+t-1}{3}\rceil)$, where $t=t_1+t_3$ and $n=t_1+3t_3$. In particular, $\chi_\ell(K_{t\times 3}) = \lceil\frac{4t-1}{3}\rceil$.
\end{thm}

Next, we have the \emph{combinatorial nullestellensatz} of Alon \cite{Alo99}, as follows.

\begin{thm}[Alon's combinatorial nullstellensatz \cite{Alo99}]\label{Alonnull}
Let $\mathbb F$ be a field, and $f(x_1,\dots,x_n)$ be a polynomial in $\mathbb F[x_1,\dots,x_n]$. Suppose the degree $\deg(f)$ of $f$ is $t_1+\cdots+t_n$, where each $t_i$ is a non-negative integer, and suppose the coefficient of $x_1^{t_1}\cdots x_n^{t_n}$ in $f$ is non-zero. Then, if $S_1,\dots,S_n$ are subsets of $\mathbb F$ with $|S_i|>t_i$, there exist $s_1\in S_1,\dots,s_n\in S_n$ so that
\[
f(s_1,\dots,s_n)\neq 0.
\]
\end{thm}

Theorem \ref{Alonnull} is useful for the study of list colourings. For instance, let $G$ be a graph, whose vertices are labelled $v_1,\dots, v_n$. For $1\le i\le n$, we associate the vertex $v_i$ with the variable $x_i$, and consider the graph polynomial
\begin{equation}
f_G(x)=\prod_{v_iv_j\in E(G),\,i<j}(x_i-x_j).\label{GPeq}
\end{equation}
Suppose that $G$ is given an $(r+1)$-list assignment $L$, and $\deg(f_G)=rn$. If we can verify that the coefficient of $x_1^r\cdots x_n^r$ in $f_G$ is non-zero, then by Theorem \ref{Alonnull}, there exist $\alpha_1\in L(v_1),\dots,\alpha_n\in L(v_n)$ such that $f_G(\alpha_1,\dots, \alpha_n)\neq 0$. The choices $\alpha_1,\dots,\alpha_n$ lead to a proper $L$-colouring of $G$. 
%
%Finally, we have the following lemma.
%
%\begin{lemma}\label{H+vlm}
%Let $H$ be a non-trivial connected graph, and $v$ be another vertex. Then $rc^\ell(H+v)\le 3$.
%\end{lemma}
%
%\begin{proof}
%We choose a spanning tree $T$ of $H$, and an ordering $e_1,\dots,e_t$ of the edges of $T$ such that for all $1\le j\le t$, the graph induced by $e_1,\dots,e_j$ is a subtree of $T$. Let $e_1=u_0u_1$, and for $2\le j\le t$, let $e_j=u_{\ell_j}u_j$, where $0\le \ell_j<j$, and $u_j$ is not a vertex of any of $e_1,\dots,e_{j-1}$. Now, let $L$ be a $3$-edge-list assignment of $H+v$. We obtain an $L$-edge-colouring $c$ as follows. Initially, let $c(u_0v)\in L(u_0v)$, $c(u_1v)\in L(u_1v)$, $c(e_1)\in L(e_1)$ be distinct. For $2\le j\le t$, suppose that $c(u_kv)\in L(u_kv)$ and $c(e_k)\in L(e_k)$ have been chosen such that $c(u_{\ell_k}v)$, $c(u_kv)$, $c(e_k)$ are distinct, for all $1\le k< j$. Let $c(u_jv)\in L(u_jv)\setminus\{c(u_{\ell_j}v)\}$ and $c(e_j)\in L(e_j)\setminus\{c(u_{\ell_j}v)\}$ be distinct. Repeat this inductively, and then assign arbitrary colours to the remaining edges of $H$ for $c$. Now for non-adjacent $u_j,u_k$ where $0\le j<k\le t$, either $u_jvu_k$ or $u_jvu_{\ell_k}u_k$ is a rainbow path. Hence, $c$ is a rainbow connected colouring of $H+v$, and $rc^\ell(H+v)\le 3$.
%\end{proof}

We are now ready to prove our results. For the cycle $C_n$, Chartrand et al.~\cite{CJMZ08} proved the following.

\begin{thm}\label{rccyclethm}\textup{\cite{CJMZ08}}
For $n\ge 4$, we have $rc(C_n)=src(C_n)=\lceil\frac{n}{2}\rceil$.
\end{thm}

We have the following result for the list colouring versions. 

\begin{thm}\label{rclcyclethm}
For $n\ge 4$, we have $rc^\ell(C_n)=src^\ell(C_n)=\lceil\frac{n}{2}\rceil$.
\end{thm}

\begin{proof}
By Theorem \ref{rccyclethm} and (\ref{ineqs}), we have $src^\ell(C_n)\ge rc^\ell(C_n)\ge rc(C_n)=\lceil\frac{n}{2}\rceil$. It suffices to prove $src^\ell(C_n)\le\lceil\frac{n}{2}\rceil$. Suppose first that $n=2r$ is even. Consider $K_{r\times 2}$, where each class of size two consists of opposite edges of $C_{2r}$. Let $L$ be an $r$-edge-list assignment of $C_{2r}$. Then $L$ becomes an $r$-list assignment of $K_{r\times 2}$. By Theorem \ref{ERTthm}(b), there exists a proper $L$-colouring $c$ of $K_{r\times 2}$. Then $c$ becomes an $L$-edge-colouring of $C_{2r}$ such that only opposite edges may possibly have the same colour. Hence, every path in $C_{2r}$ with length at most $r$ must be rainbow coloured under $c$, and $src^\ell(C_{2r})\le r$.

Now, suppose $n=2r+1\ge 5$ is odd. Let $L$ be an $(r+1)$-edge-list assignment of $C=C_{2r+1}$. Let $e\in E(C)$ and $\alpha\in L(e)$. Consider the cycle $C'=C_{2r}$ by contracting the edge $e$, and the $r$-edge-list assignment $L'$ of $C'$ where $L'(f)=L(f)\setminus\{\alpha\}$ for all $f\in E(C')$. By the result for $C_{2r}$, there exists an $L'$-edge-colouring $c'$ of $C'$ such that every path in $C'$ with length at most $r$ is rainbow coloured under $c'$. Letting $c'(e)=\alpha$, we have $c'$ is an $L$-edge-colouring of $C$ such that every path in $C$ with length at most $r$ is rainbow coloured under $c'$. Hence, $src^\ell(C_{2r+1})\le r+1$.
\end{proof}

Next, let $C_n=v_0v_1\cdots v_{n-1}v_0$ be the cycle on $n\ge 3$ vertices, where the indices of the vertices are taken modulo $n$. The \emph{wheel} on $n+1$ vertices is the graph $W_n$, obtained by connecting another vertex $v$ to all vertices of $C_n$. Chartrand et al.~\cite{CJMZ08} proved the following.
\begin{thm}\label{rcwheelthm}\textup{\cite{CJMZ08}}
Let $n\ge 3$. Then
\begin{enumerate}
\item[(a)] 
\[
rc(W_n)=
\left\{
\begin{array}{ll}
1, & \textup{\emph{if} }n=3,\\[0.3ex]
2, & \textup{\emph{if} }n=4,5,6,\\[0.3ex]
3, & \textup{\emph{if} }n\ge 7.
\end{array}
\right.
\]
\item[(b)] $src(W_n)=\lceil\frac{n}{3}\rceil$.
\end{enumerate}
\end{thm}
For list colourings, we first have the result that $rc^\ell(W_n)$ is equal to $rc(W_n)$.
\begin{thm}\label{rclwheelthm}
Let $n\ge 3$. Then
\[
rc^\ell(W_n)=
\left\{
\begin{array}{ll}
1, & \textup{\emph{if} }n=3,\\[0.3ex]
2, & \textup{\emph{if} }n=4,5,6,\\[0.3ex]
3, & \textup{\emph{if} }n\ge 7.
\end{array}
\right.
\]
\end{thm}

\begin{proof}
Clearly, $rc^\ell(W_3)=rc^\ell(K_4)=1$. By Theorem \ref{rcwheelthm}(a) and (\ref{ineqs}), we have $rc^\ell(W_n)\ge 2$ for $4\le n\le 6$, and $rc^\ell(W_n)\ge 3$ for $n\ge 7$. By Theorem \ref{univthm}(b), we have $rc^\ell(W_n)\le 3$ for $n\ge 3$. Now, let $4\le n\le 6$. Let $L$ be a $2$-edge-list assignment of $W_n$. We obtain an $L$-edge-colouring $c$ as follows. By Theorem \ref{ERTthm}(a), we have $\chi'_\ell(C_{2p})=2$ for $p\ge 2$. For $n=4,6$, choose a proper $L$-edge-colouring $c$ for $C_n$. In addition, for $n=6$, let $c(vv_i)\in L(vv_i)$ and $c(vv_{i+3})\in L(vv_{i+3})$ be distinct for $0\le i\le 2$. For $n=5$, let $c$ be an $L$-edge-colouring such that $c(v_iv_{i+1})\neq c(v_{i+1}v_{i+2})$ for $0\le i\le 3$, and $c(vv_1)\neq c(vv_4)$. For the remaining edges, choose arbitrary colours from their lists. Then $c$ is a rainbow connected colouring, and $rc^\ell(W_n)\le 2$.
%\indent Now, let $n\ge 7$. Let $L$ be a $3$-edge-list assignment of $W_n$. We obtain an $L$-edge-colouring $c$ as follows. Firstly, choose the colours for the edges at $v$ so that $c(vv_i)\neq c(vv_{i+1})$ for $0\le i\le n-1$. Then, choose the colours for $C_n$ so that $c(v_iv_{i+1})\not\in\{c(vv_i),c(vv_{i+1})\}$ for $0\le i\le n-1$. Now, two non-adjacent vertices of $W_n$ must be some $v_i$ and $v_j$  such that $v_iv_j\not\in E(C_n)$. If $c(vv_i)\neq c(vv_j)$, then $v_ivv_j$ is a rainbow path. If $c(vv_i)= c(vv_j)$, then $v_ivv_{j-1}v_j$ is a rainbow path. Thus, $c$ is a rainbow connected $L$-edge-colouring, and $rc^\ell(W_n)\le 3$.
\end{proof}
For the strong rainbow connection parameters, we have not been able to prove whether or not we have $src^\ell(W_n)=src(W_n)$. We have the following partial result.
\begin{thm}\label{srclwheelthm}
We have\\[-0.2cm]
\[
src^\ell(W_n)=
\left\{
\begin{array}{ll}
1, & \textup{\emph{if} }n=3,\\[0.3ex]
2, & \textup{\emph{if} }n=4,5,6,\\[0.3ex]
3, & \textup{\emph{if} }n=7,8,9.
\end{array}
\right.
\]
\indent For $n\ge 10$, we have 
\begin{equation}\label{srclWnUB}
\Big\lceil\frac{n}{3}\Big\rceil \le src^\ell(W_n)\le
\left\{
\begin{array}{ll}
\lceil\frac{4n-3}{9}\rceil, & \textup{\emph{if} }n\equiv 0\textup{ (mod 3)},\\[0.3ex]
\lceil\frac{4n-1}{9}\rceil, & \textup{\emph{if} }n\equiv 1\textup{ (mod 3)},\\[0.3ex]
\lceil\frac{4n+1}{9}\rceil, & \textup{\emph{if} }n\equiv 2\textup{ (mod 3)}.
\end{array}
\right.
\end{equation}
\end{thm}

We first show in Proposition \ref{Cn2compprop} below that for $n\ge 7$, computing $src^\ell(W_n)$ is equivalent to computing the list chromatic number of $\overline{C_n^2}$, the complement of the square of the cycle $C_n$. The latter problem can be seen to be more natural.  Recall that $C_n^2$ is obtained from $C_n$ by adding all edges $xy$ where $x$ and $y$ are at distance $2$ in $C_n$. Let $u_0,\dots,u_{n-1}$ be the vertices of $\overline{C_n^2}$, taken cyclically modulo $n$ in that order.

\begin{prop}\label{Cn2compprop}
Let $n\ge 7$. Then $src^\ell(W_n)=\chi_\ell(\overline{C_n^2})$.
\end{prop}

\begin{proof}
Write $C$ for the cycle $C_n$ of $W_n$, and $C_n'$ for $\overline{C_n^2}$. First, let $s=\chi_\ell(C_n')$, and $L$ be an $s$-edge-list assignment of $W_n$. We have an $s$-list assignment $L'$ of $C_n'$, where $L'(u_i)=L(vv_i)$ for $0\le i\le n-1$. There exists a proper $L'$-colouring of $C_n'$. This leads a colouring $c$ of the edges $vv_i$ of $W_n$, where $c(vv_j)\neq c(vv_k)$ whenever $d_C(v_j,v_k)\ge 3$. Note that $s=\chi_\ell(C_n')\ge \chi(C_n')\ge 3$, since $C_7'=C_7$; $C_8'\supset C_5$; and $C_n'\supset C_3$ if $n\ge 9$. Thus, $c$ becomes a strongly rainbow connected $L$-edge-colouring of $W_n$ if we set $c$ to be proper on $C$. Hence, we have $src^\ell(W_n)\le \chi_\ell(C_n')$.

Now, let $r=src^\ell(W_n)$, and $L$ be an $r$-list assignment of $C_n'$. We have an $r$-edge-list assignment $L'$ of the edges $vv_i$ of $W_n$, where $L'(vv_i)=L(u_i)$ for $0\le i\le n-1$. We extend $L'$ to all edges of $W_n$, by assigning arbitrary lists of size $r$ to the edges of $C$. There exists a strongly rainbow connected $L'$-edge-colouring $c$ of $W_n$. Since $c(vv_j)\neq c(vv_k)$ whenever $d_C(v_j,v_k)\ge 3$, we may obtain a proper $L$-colouring of $C_n'$ from $c$. Hence, we have $src^\ell(W_n)\ge \chi_\ell(C_n')$.
\end{proof}

\begin{proof}[Proof of Theorem \ref{srclwheelthm}]
As before, write $C$ for the cycle $C_n$ of $W_n$, and $C_n'$ for $\overline{C_n^2}$. By Theorem \ref{rcwheelthm}(b) and (\ref{ineqs}), we have $src^\ell(W_n)\ge src(W_n)=\lceil\frac{n}{3}\rceil$ for all $n\ge 3$. Clearly, $src^\ell(W_3)=src^\ell(K_4)=1$. For $4\le n\le 6$, given a $2$-edge-list assignment $L$ of $W_n$, the $L$-edge-colouring as described in the proof of Theorem \ref{rclwheelthm} is in fact strongly rainbow connected. Hence, $src^\ell(W_n)=2$. We prove the upper bounds for $src^\ell(W_n)$, where $n\ge 7$. %For $t\ge 3$, we have $C_{3t-2}'\subset C_{3t-1}'\subset C_{3t}'$, and thus $\chi_\ell(C_{3t-2}')\le\chi_\ell(C_{3t-1}')\le\chi_\ell(C_{3t}')$. 

First, consider $7\le n\le 9$. Since $C_7'\subset C_8'\subset C_9'$, we have $\chi_\ell(C_7')\le\chi_\ell(C_8')\le\chi_\ell(C_9')$. By Proposition \ref{Cn2compprop}, it suffices to prove $\chi_\ell(C_9')\le 3$. Consider the graph polynomial $f_{C_9'}$ of $C_9'$ as defined in (\ref{GPeq}), where the vertex $u_i$ is associated with the variable $x_i$, for $0\le i\le 8$. We have $f_{C_9'}$ is a product of 18 terms of the form $(x_i-x_j)$, where $0\le i<j\le 8$ and $u_iu_j\in E(C_9')$, and so $\deg(f_{C_9'})=18$. With some effort, we may compute the coefficient of $x_0^2\cdots x_8^2$ in $f_{C_9'}$, which is $-18$. By Alon's combinatorial nullstellensatz (Theorem \ref{Alonnull}), if $L$ is a $3$-list assignment of $C_9'$, there exist $\alpha_0\in L(u_0),\dots,\alpha_8\in L(u_8)$ such that $f_{C_9'}(\alpha_0,\dots,\alpha_8)\neq 0$. The choices $\alpha_0,\dots,\alpha_8$ lead to a proper $L$-colouring of $C_9'$, and hence $\chi_\ell(C_9')\le 3$, as required.

Now, let $n=t_1+3t_3\ge 10$, where $0\le t_1\le 2$. Note that $C_n'\subset K_{t_1\times 1,t_3\times 3}$, where the classes of $K_{t_1\times 1,t_3\times 3}$ are $\{u_0,u_1,u_2\},\{u_3,u_4,u_5\},\dots$, with $\{u_{n-1}\}$ if $n\equiv 1$ (mod 3), and $\{u_{n-2}\},\{u_{n-1}\}$ if $n\equiv 2$ (mod 3). By Theorem \ref{KieOhbthm} and Proposition \ref{Cn2compprop}, we have 
\[
src^\ell(W_n)=\chi_\ell(C_n')\le \chi_\ell(K_{t_1\times 1,t_3\times 3})=\max\Big(t,\Big\lceil\frac{n+t-1}{3}\Big\rceil\Big),
\]
where $t=t_1+t_3$ and $n=t_1+3t_3$. We have $t\le \lceil\frac{n+t-1}{3}\rceil$, and so $src^\ell(W_n)\le\lceil\frac{n+t-1}{3}\rceil=\lceil\frac{4n+2t_1-3}{9}\rceil$. The upper bounds of (\ref{srclWnUB}) follow.
\end{proof}

In Theorem \ref{srclwheelthm}, we believe that the correct value of $src^\ell(W_n)$ should be far from $\frac{4n}{9}$. We leave the following as an open problem.

\begin{prob}\label{srclwheelprob}
Determine $src^\ell(W_n)$, for all $n\ge 3$.
\end{prob}

Now, we consider complete bipartite graphs. Let $K_{m,n}$ be the complete bipartite graph with classes $U=\{u_1,\dots,u_m\}$ and $V=\{v_1,\dots,v_n\}$, where $1\le m\le n$. For an edge-colouring $c$ of $K_{m,n}$, we associate each vertex $v_j$ with an $m$-vector $\vec{v}_j$, where the $i$th component is $\vec{v}_{ji}=c(v_iv_j)$, for $1\le i\le m$.

The rainbow connection and strong rainbow connection numbers of $K_{m,n}$ were determined by Chartrand et al.~\cite{CJMZ08}, as follows.
\begin{thm}\label{rcsrcKmnthm}\textup{\cite{CJMZ08}}
Let $1\le m\le n$. Then 
\begin{enumerate}
\item[(a)]
\[
rc(K_{m,n})=
\left\{
\begin{array}{ll}
n, & \textup{\emph{if }}m=1,\\
\min(\lceil\!\sqrt[m]{n}\,\rceil,4), & \textup{\emph{if }}m\ge 2.
\end{array}
\right.
\]
\item[(b)] $src(K_{m,n})=\lceil\!\sqrt[m]{n}\,\rceil$.
\end{enumerate}
\end{thm}

We prove that the results of Theorem \ref{rcsrcKmnthm} remain true for the list colouring analogues.

\begin{thm}\label{rcsrclKmnthm}
Let $1\le m\le n$. Then 
\begin{enumerate}
\item[(a)]
\[
rc^\ell(K_{m,n})=
\left\{
\begin{array}{ll}
n, & \textup{\emph{if }}m=1,\\
\min(\lceil\!\sqrt[m]{n}\,\rceil,4), & \textup{\emph{if }}m\ge 2.
\end{array}
\right.
\]
\item[(b)] $src^\ell(K_{m,n})=\lceil\!\sqrt[m]{n}\,\rceil$.
\end{enumerate}
\end{thm}

\begin{proof}
By Theorem \ref{charthm}(g), we have $rc^\ell(K_{1,n})=src^\ell(K_{1,n})=n$. Let $2\le m\le n$.\\[1ex]
\indent(b) By Theorem \ref{rcsrcKmnthm}(b) and (\ref{ineqs}), it suffices to prove $src^\ell(K_{m,n})\le\lceil\!\sqrt[m]{n}\,\rceil$. Let $r=\lceil\!\sqrt[m]{n}\,\rceil\ge 2$, and $L$ be an $r$-edge-list assignment of $K_{m,n}$. We construct an $L$-edge-colouring by choosing the associated vectors $\vec{v}_j$, for $1\le j\le n$, which satisfy the following properties.
\begin{enumerate}
\item[(i)] $\vec{v}_1,\dots,\vec{v}_n$ are distinct.
%\item[(ii)] $\vec{v}_{ji}\in L(u_iv_j)$ for all $1\le i\le m$ and $1\le j\le n$;
\item[(ii)] $\vec{v}_{j,j+1},\dots,\vec{v}_{jm}$ are different from $\vec{v}_{jj}$, for all $1\le j\le m$.
\end{enumerate}
Observe that without the properties (i) and (ii), there are at least $r^m$ possibilities for each $\vec{v}_j$, which lie in the $m$-dimensional grid $L(u_1v_j)\times\cdots\times L(u_mv_j)$. We first choose $\vec{v}_1,\dots,\vec{v}_m$. For $\vec{v}_1$, let $\vec{v}_{1i}\in L(u_iv_1)$ for $1\le i\le m$ so that $\vec{v}_{12},\dots,\vec{v}_{1m}$ are different from $\vec{v}_{11}$. Now, suppose for some $2\le j\le m$, we have chosen distinct $\vec{v}_1,\dots,\vec{v}_{j-1}$ such that $\vec{v}_{hi}\in L(u_iv_h)$, and $\vec{v}_{h,h+1},\dots,\vec{v}_{hm}$ are different from $\vec{v}_{hh}$, for all $1\le h\le j-1$ and $1\le i\le m$. We choose $\vec{v}_j$. Let
$\vec{v}_{ji}\in L(u_iv_j)$ for $1\le i\le m$ so that $\vec{v}_{j,j+1},\dots,\vec{v}_{jm}$ are different from $\vec{v}_{jj}$. We may furthermore choose $\vec{v}_j$ in this way so that $\vec{v}_1,\dots,\vec{v}_{j-1}$ are different from $\vec{v}_j$, since there are at least $r^j(r-1)^{m-j}-(j-1)\ge r^j-j+1\ge 1$ possible choices for $\vec{v}_j$. Continuing inductively, we have chosen the vectors $\vec{v}_1,\dots,\vec{v}_m$.

Now, we choose $\vec{v}_{m+1},\dots,\vec{v}_n$. Suppose for some $m+1\le j\le n$, we have chosen distinct $\vec{v}_1,\dots,\vec{v}_{j-1}$ such that $\vec{v}_{hi}\in L(u_iv_h)$ for all $1\le h\le j-1$ and $1\le i\le m$, with $\vec{v}_1,\dots,\vec{v}_m$ as above. There are at least $r^m-(j-1)\ge r^m-n+1\ge 1$ possible choices for $\vec{v}_j$ such that, $\vec{v}_{ji}\in L(u_iv_j)$ for all $1\le i\le m$, and $\vec{v}_1,\dots,\vec{v}_{j-1}$ are different from $\vec{v}_j$. Continuing inductively, we have chosen the vectors $\vec{v}_1,\dots,\vec{v}_n$, which satisfy the properties (i) and (ii).

For this $L$-edge-colouring of $K_{m,n}$, if $u_i,u_p\in U$ where $1\le i<p\le m$, then $u_iv_iu_p$ is a rainbow path since $\vec{v}_{ii}\neq\vec{v}_{ip}$ by (ii). If $v_j,v_q\in V$ where $1\le j<q\le n$, then since $\vec{v}_j\neq\vec{v}_q$ by (i), we have $\vec{v}_{ji}\neq\vec{v}_{qi}$ for some $1\le i\le m$, and $v_ju_iv_q$ is a rainbow path. Hence, we have a strongly rainbow connected $L$-edge-colouring of $K_{m,n}$, and $src^\ell(K_{m,n})\le r=\lceil\!\sqrt[m]{n}\,\rceil$.\\[1ex]
\indent(a) By Theorem \ref{rcsrcKmnthm}(a) and (\ref{ineqs}), it suffices to prove $rc^\ell(K_{m,n})\le\min(\lceil\!\sqrt[m]{n}\,\rceil,4)$. By (b) and (\ref{ineqs}), we have $rc^\ell(K_{m,n})\le src^\ell(K_{m,n})\le\lceil\!\sqrt[m]{n}\,\rceil$. We are done if $\lceil\!\sqrt[m]{n}\,\rceil\le 3$, or equivalently, $m\le n\le 3^m$. Now let $\lceil\!\sqrt[m]{n}\,\rceil\ge 4$, or equivalently, $n\ge 3^m+1$. We prove $rc^\ell(K_{m,n})\le 4$. Let $L$ be a $4$-edge-list assignment of $K_{m,n}$. We construct an $L$-edge-colouring by choosing the associated vectors $\vec{v}_j$, for $1\le j\le n$. Let $N=3^m$, and $V'=\{v_1,\dots,v_N\}$. We take the strongly rainbow connected $L$-edge-colouring of $K_{m,N}$ with bipartition $(U,V')$ as described in (b). We have $\vec{v}_{11}\neq\vec{v}_{12}$ from (ii). Let $\vec{v}_{ji}\in L(u_iv_j)$ be such that $\vec{v}_{11},\vec{v}_{12},\vec{v}_{j1},\vec{v}_{j2}$ are distinct, for all $1\le i\le m$ and $N+1\le j\le n$.

For this $L$-edge-colouring of $K_{m,n}$, let $v_j,v_q\in V$ where $1\le j\le n$ and $N+1\le q\le n$. If $\vec{v}_j\neq\vec{v}_q$, then $\vec{v}_{ji}\neq\vec{v}_{qi}$ for some $1\le i\le m$, and $v_ju_iv_q$ is a rainbow path. Otherwise, we have $\vec{v}_j=\vec{v}_q$, and $v_j\neq v_1$ since $\vec{v}_{11}\neq \vec{v}_{q1}$. The path $v_ju_1v_1u_2v_q$ is rainbow, since the four colours of this path are $\vec{v}_{j1},\vec{v}_{11},\vec{v}_{12}$ and $\vec{v}_{q2}=\vec{v}_{j2}$. Hence, we have a rainbow connected $L$-edge-colouring of $K_{m,n}$, and $rc^\ell(K_{m,n})\le 4$.
\end{proof}
Next, we consider complete multipartite graphs. Let $t\ge 3$, and $K_{n_1,\dots,n_t}$ denote the complete $t$-partite graph with class sizes $1\le n_1\le\cdots\le n_t$. Let $m=\sum_{i=1}^{t-1}n_i$ and $n=n_t$. In this way, we will consider the relation between the graphs $K_{n_1,\dots,n_t}$ and $K_{m,n}$.

The rainbow connection and strong rainbow connection numbers of $K_{n_1,\dots,n_t}$ have been determined by Chartrand et al.~\cite{CJMZ08}, as follows.
\begin{thm}\label{rcsrcmpthm}\textup{\cite{CJMZ08}}
Let $t\ge 3$ and $1\le n_1\le\cdots\le n_t$. Let $m=\sum_{i=1}^{t-1}n_i$ and $n=n_t$. Then 
\begin{enumerate}
\item[(a)]
\[
rc(K_{n_1,\dots,n_t})=
\left\{
\begin{array}{ll}
1, & \textup{\emph{if }}n_t=1,\\
2, & \textup{\emph{if }}n_t\ge 2\textup{\emph{ and }}m>n,\\
\min(\lceil\!\sqrt[m]{n}\,\rceil,3), & \textup{\emph{if }}m\le n.
\end{array}
\right.
\]
\item[(b)]
\[
src(K_{n_1,\dots,n_t})=
\left\{
\begin{array}{ll}
1, & \textup{\emph{if }}n_t=1,\\
2, & \textup{\emph{if }}n_t\ge 2\textup{\emph{ and }}m>n,\\
\lceil\!\sqrt[m]{n}\,\rceil, & \textup{\emph{if }}m\le n.
\end{array}
\right.
\]
\end{enumerate}
\end{thm}

We prove that the results of Theorem \ref{rcsrcmpthm} remain true for the list colouring analogues.
\begin{thm}\label{rcsrclmpthm}
Let $t\ge 3$ and $1\le n_1\le\cdots\le n_t$. Let $m=\sum_{i=1}^{t-1}n_i$ and $n=n_t$. Then 
\begin{enumerate}
\item[(a)]
\[
rc^\ell(K_{n_1,\dots,n_t})=
\left\{
\begin{array}{ll}
1, & \textup{\emph{if }}n_t=1,\\
2, & \textup{\emph{if }}n_t\ge 2\textup{\emph{ and }}m>n,\\
\min(\lceil\!\sqrt[m]{n}\,\rceil,3), & \textup{\emph{if }}m\le n.
\end{array}
\right.
\]
\item[(b)]
\[
src^\ell(K_{n_1,\dots,n_t})=
\left\{
\begin{array}{ll}
1, & \textup{\emph{if }}n_t=1,\\
2, & \textup{\emph{if }}n_t\ge 2\textup{\emph{ and }}m>n,\\
\lceil\!\sqrt[m]{n}\,\rceil, & \textup{\emph{if }}m\le n.
\end{array}
\right.
\]
\end{enumerate}
\end{thm}

\begin{proof}
Write $G=K_{n_1,\dots,n_t}$. If $n_t=1$, then $G=K_t$, and Theorem \ref{charthm}(e) gives $rc^\ell(G)=src^\ell(G)=1$. 

Now, let $n_t\ge 2$. By Theorem \ref{rcsrcmpthm} and (\ref{ineqs}), it suffices to prove $src^\ell(G)\le 2$ if $m>n$; and $rc^\ell(G)\le \min(\lceil\!\sqrt[m]{n}\,\rceil,3)$  and $src^\ell(G)\le \lceil\!\sqrt[m]{n}\,\rceil$, if $m\le n$. We obtain a spanning complete bipartite subgraph $H\subset G$ as follows. If $m\le n$, we let $H=K_{m,n}$. For $m>n$, we let $H=K_{a,b}$ such that, each class is a union of some of the classes of $G$, and $a\le b$ such that $b-a$ is minimum. In \cite{CJMZ08}, Theorem 2.5, it was proved that $b\le 2^a$, or equivalently, $\lceil\sqrt[a]{b}\,\rceil=2$.

Let $r=2$ if $m>n$, and $r=\lceil\!\sqrt[m]{n}\,\rceil$ if $m\le n$. Let $L$ be an $r$-edge-list assignment of $G$. By restricting to the subgraph $H$, we have an $r$-edge-list assignment $L'$ of $H$. Since $\lceil\sqrt[a]{b}\,\rceil=2$ when $m>n$, by Theorem \ref{rcsrclKmnthm}(b), there exists a strongly rainbow connected $L'$-edge-colouring $c'$ of $H$. Let $c$ be an $L$-edge-colouring of $G$ such that the restriction of $c$ to $H$ is $c'$. Then $c$ is a strongly rainbow connected colouring of $G$, since any two vertices of $G$ in the same class are connected by a rainbow path of length $2$ under $c'$, and hence under $c$. Thus, we have $rc^\ell(G)\le src^\ell(G)\le r$.

It remains to prove that $rc^\ell(G)\le 3$ for $m\le n$ and $\lceil\!\sqrt[m]{n}\,\rceil\ge 3$. Equivalently, $n\ge 2^m+1$. Let $L$ be a $3$-edge-list assignment of $G$. We construct an $L$-edge-colouring of $G$. Let $L'$ be the restriction of $L$ to the subgraph $K_{m,n}$ of $G$, with classes $U=\{u_1,\dots,u_m\}$ and $V=\{v_1,\dots,v_n\}$. Since $t\ge 3$, we may assume $u_1$ and $u_2$ are in different classes of $G$. Let $N=2^m$, and $V'=\{v_1,\dots,v_N\}$. We take the strongly rainbow connected $L'$-edge-colouring of $K_{m,N}$ with bipartition $(U,V')$ as described in Theorem \ref{rcsrclKmnthm}(b). Choose $\alpha\in L(u_1u_2)$ for $u_1u_2$, and $\vec{v}_{ji}\in L(u_iv_j)$ such that $\vec{v}_{j1},\vec{v}_{j2},\alpha$ are distinct, for all $1\le i\le m$ and $N+1\le j\le n$. Choose arbitrary colours for the remaining edges of $G$, from their lists under $L$.

For this $L$-edge-colouring of $G$, let $v_j,v_q\in V$ where $1\le j\le n$ and $N+1\le q\le n$. If $\vec{v}_j\neq\vec{v}_q$, then $\vec{v}_{ji}\neq\vec{v}_{qi}$ for some $1\le i\le m$, and $v_ju_iv_q$ is a rainbow path. Otherwise, we have $\vec{v}_j=\vec{v}_q$. The path $v_ju_1u_2v_q$ is rainbow, since the three colours of this path are $\vec{v}_{j1},\alpha$ and $\vec{v}_{q2}=\vec{v}_{j2}$. Hence, we have a rainbow connected $L$-edge-colouring of $G$, and $rc^\ell(G)\le 3$.
\end{proof}

We conclude this section by considering the Petersen graph ${\textsf P}_{10}$, with the drawing as shown in Figure 2. Chartrand et al.~\cite{CJMZ08} computed $rc({\textsf P}_{10})$ and $src({\textsf P}_{10})$.
\begin{thm}\label{CJMZPetthm}\textup{\cite{CJMZ08}}
$rc(\textup{\textsf P}_{10})=3$ and $src(\textup{\textsf P}_{10})=4$. 
\end{thm}
\indent\\[-1.05cm]
\begin{figure}[htp]
\centering
\includegraphics[width=5cm]{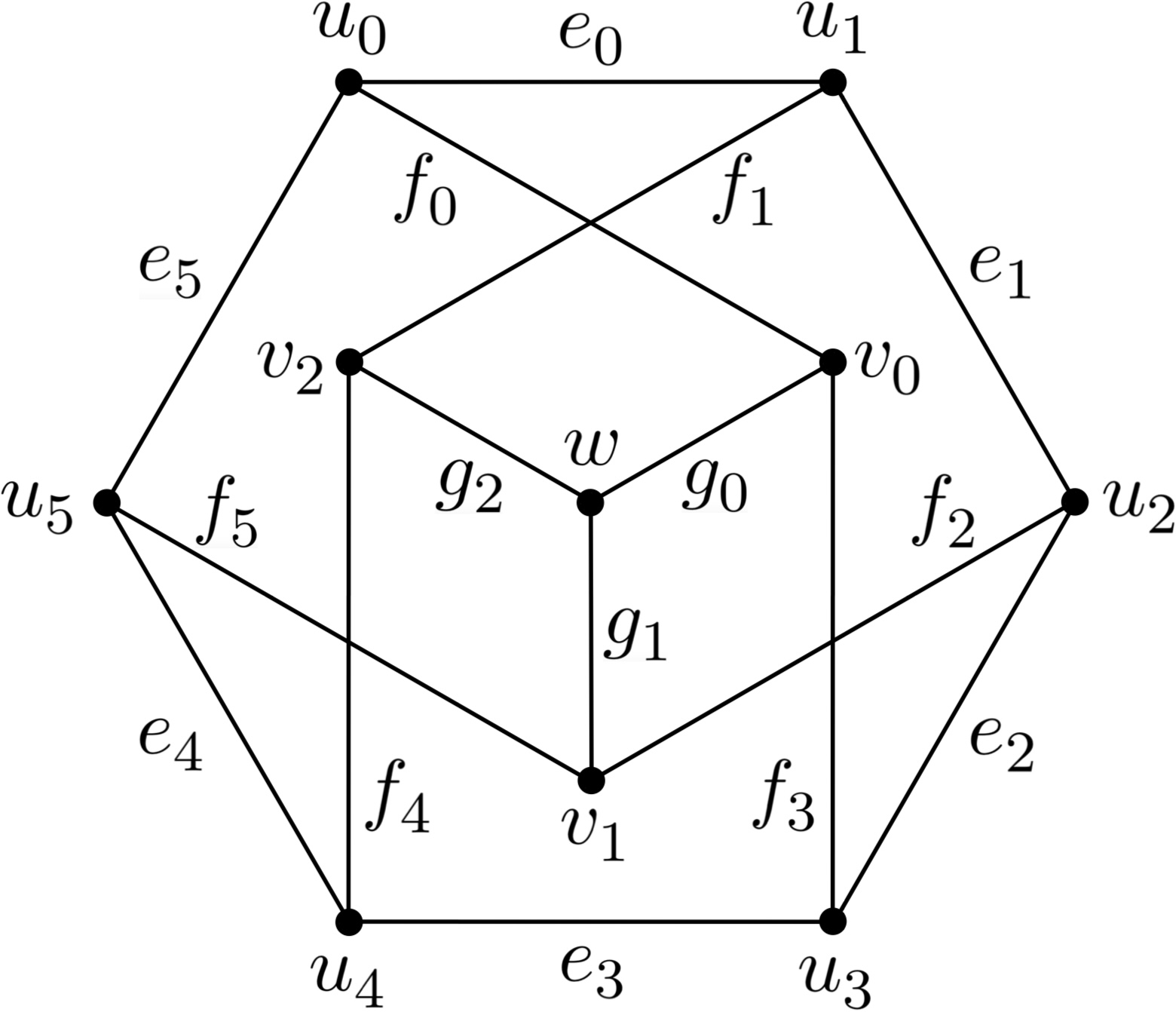}
\end{figure}
\indent\\[-0.9cm]
\begin{center}
\hspace{0.02cm}Figure 2. The Petersen graph ${\textsf P}_{10}$
\end{center}
\indent\\[-0.6cm]
%\noindent $u_0\quad u_1\quad u_2\quad u_3\quad u_4$\\
%$u_5\quad v_0\quad v_1\quad v_2\quad w$\\
%$e_0\quad e_1\quad e_2\quad e_3\quad e_4$\\
%$f_0\quad f_1\quad f_2\quad f_3\quad f_4$\\
%$g_0\quad g_1\quad g_2\quad e_5\quad f_5$

We prove that the same results hold for the list colouring analogues.

\begin{thm}\label{Petthm}
$rc^\ell(\textup{\textsf P}_{10})=3$ and $src^\ell(\textup{\textsf P}_{10})=4$. 
\end{thm}

\begin{proof}
By Theorem \ref{CJMZPetthm} and (\ref{ineqs}), we have $rc^\ell(\textup{\textsf P}_{10})\ge 3$ and $src^\ell(\textup{\textsf P}_{10})\ge 4$. We prove the matching upper bounds.

It is well-known that $\chi'(\textup{\textsf P}_{10})=4$. A result of Schauz \cite{Sch18} implies that $\textup{\textsf P}_{10}$ satisfies the list colouring conjecture, and thus $\chi_\ell'(\textup{\textsf P}_{10})=4$. Since diam$(\textup{\textsf P}_{10})=2$, this means that a proper edge-colouring of $\textup{\textsf P}_{10}$ is strongly rainbow connected. It follows that $src^\ell(\textup{\textsf P}_{10})\le 4$.

Now, let $u_p,v_q,w$ be the vertices and $e_i,f_j,g_k$ be the edges of $\textsf P_{10}$ as shown in Figure 2, where all indices related to the $u_p,e_i,f_j$ are taken modulo $6$. Let $L$ be a $3$-edge-list assignment of $\textup{\textsf P}_{10}$. We construct an $L$-edge-colouring $c$ as follows. First, as in proof of Theorem \ref{rclcyclethm}, we may choose $c(e_i)=\alpha_i\in L(e_i)$ for $0\le i\le 5$ such that only opposite edges of the $C_6$ may possibly have the same colour. Then, we choose $c(f_j)=\beta_j\in L(f_j)$ with $\beta_j\neq\alpha_j,\alpha_{j-1}$ for $0\le j\le 5$; and $c(g_k)=\gamma_k\in L(g_k)$ with $\gamma_k\neq\beta_{2k},\beta_{2k+3}$ for $0\le k\le 2$. Now, any two vertices are connected by a rainbow path under $c$, except possibly for two of $v_0,v_1,v_2$. If $\gamma_0,\gamma_1,\gamma_2$ are distinct, then $c$ is rainbow connected. Otherwise, we consider two cases.\\[1ex]
\emph{Case 1.} Without loss of generality, $\gamma_0=\gamma_1=\gamma$, where $\gamma\neq\gamma_2$.\\[1ex]
\indent If there exists a rainbow $v_0-v_1$ path, then $c$ is rainbow connected. Otherwise, we have $\beta_2=\beta_3=\beta$ and $\beta_0=\beta_5=\delta$. Note that $\beta=\delta$, $\gamma=\alpha_1$, $\gamma=\alpha_2$ are possible. If we can update $c(f_2)\in L(f_2)$ so that $c(f_2)\neq\alpha_1,\alpha_2,\beta,\gamma$, then $v_0u_3u_2v_1$ is a rainbow $v_0-v_1$ path, and $c$ is rainbow connected. Otherwise, the only possible elements of $L(f_2)$ are $\alpha_1,\alpha_2,\beta,\gamma$. We update $c(f_2)=\alpha_1$ if possible. Otherwise, we have $L(f_2)=\{\alpha_2,\beta,\gamma\}$ and $\gamma\neq\alpha_2$, and we update $c(f_2)=\gamma$. Then $v_0u_3u_2v_1$ is a rainbow $v_0-v_1$ path. Also, $u_1u_0u_5v_1$ is a rainbow $u_1-v_1$ path, since $\alpha_0\neq\delta$. Finally, for a rainbow $u_2-w$ path, we take $u_2v_1w$ or $u_2u_3v_0w$ if $c(f_2)=\alpha_1$; and $u_2u_3v_0w$ if $c(f_2)=\gamma$. All other pairs of vertices are connected by a rainbow path.  Hence, the updated colouring $c$ is rainbow connected.\\[1ex]
\emph{Case 2.} $\gamma_0=\gamma_1=\gamma_2=\gamma$.\\[1ex]
\indent Suppose that we cannot reduce to Case 1 by updating $c(g_0)$, $c(g_1)$, $c(g_2)$. Then, we have $L(g_k)=\{\beta_{2k},\beta_{2k+3},\gamma\}$ for $0\le k\le 2$. Now, the six equalities $\beta_2=\beta_3$, $\beta_0=\beta_5$, $\beta_4=\beta_5$, $\beta_2=\beta_1$, $\beta_0=\beta_1$, $\beta_4=\beta_3$ cannot all hold simultaneously, otherwise we have $\beta_0=\beta_3$, a contradiction. We may assume that either $\beta_4\neq\beta_5$ or $\beta_2\neq\beta_1$, so that there exists a rainbow $v_1-v_2$ path of length $3$. Now, we update $c(g_0)=\beta_0$. We have either $u_0u_1v_2w$ or $u_0u_5v_1w$ is a rainbow $u_0-w$ path. All other pairs of vertices are connected by a rainbow path.  Hence, the updated colouring $c$ is rainbow connected.\\[1ex]
\indent We now have $rc^\ell(\textup{\textsf P}_{10})\le 3$. This completes the proof of Theorem \ref{Petthm}. 
\end{proof}

\section{Rainbow connection numbers with prescribed values}\label{compsect}

In this section, we consider how the various rainbow connection parameters that we have seen can attain prescribed values. In view of the inequality $rc(G) \le src(G)$ for any connected graph $G$, Chartrand et al.~\cite{CJMZ08} considered the following question: \emph{Given positive integers $a \le b$, does there exist a graph $G$ such that $rc(G) = a$ and $src(G) = b$?} They gave positive
answers for $a=b$, and $3\le a<b$ with $b\ge \frac{5a-6}{3}$. Chern and Li \cite{CL12} improved this result as follows.

\begin{thm}\label{CLthm}\textup{\cite{CL12}}
Let $a$ and $b$ be positive integers. Then there exists a connected graph $G$ such that $rc(G) = a$ and $src(G) = b$ if and only if $a = b \in \{1,2\}$ or $3 \le a \le b$.
\end{thm}
Theorem \ref{CLthm} was an open problem of Chartrand et al., and it completely characterises all possible pairs $a$ and $b$ for the above question. Subsequently, Chen et al.~\cite{CLLL18} proved some similar results for the vertex-coloured and total-coloured versions of the rainbow connection parameter. Here, in view of the inequalities $src(G)\le src^\ell(G)$ and $rc^\ell(G)\le src^\ell(G)$ of (\ref{ineqs}), we have Theorems \ref{srcsrclcompthm} and \ref{rclsrclcompthm} below.

\begin{thm}\label{srcsrclcompthm}
Let $a$ and $b$ be positive integers. Then there exists a connected graph $G$ such that $src(G) = a$ and $src^\ell(G) = b$ if and only if $a=b=1$ or $2\le a\le b$.
\end{thm}

\begin{thm}\label{rclsrclcompthm}
Let $a$ and $b$ be positive integers. Then there exists a connected graph $G$ such that $rc^\ell(G) = a$ and $src^\ell(G) = b$ if and only if $a=b=1$ or $2\le a\le b$.
\end{thm}

Before we prove Theorems \ref{srcsrclcompthm} and \ref{rclsrclcompthm}, we prove some lemmas.

\begin{lemma}\label{preslem1}
Let $b\ge 2$. Let $G$ be a graph on $b+(b-1)^{b-1}$ vertices as follows: Take vertex-disjoint graphs $H$ and $K$ on $b-1$ and $(b-1)^{b-1}$ vertices, where $K$ is a clique, and connect another vertex $v$ to all vertices of $H$ and $K$. Then $src^\ell(G)=b$.
\end{lemma}

\begin{proof}
This lemma is inspired by the fact that $\chi_\ell(K_{p,p^p})=p+1$ for $p\ge 1$. First, we define the $(b-1)$-edge-list assignment $L$ of $G$ as follows. Over $u\in V(H)$, let $L(uv)$ be disjoint sets of size $b-1$. Then, over $w\in V(K)$, let $L(wv)$ be all the possible sets of size $b-1$ by taking one element from each $L(uv)$, for $u\in V(H)$. All remaining edges are given arbitrary lists of size $b-1$. Now, let $f$ be an $L$-edge-colouring. Then there exists $y\in V(K)$ such that $L(yv)=\{f(uv):u\in V(H)\}$. Thus $f(yv)=f(xv)$ for some $x\in V(H)$, and there is no rainbow geodesic connecting $x$ and $y$. Hence, $src^\ell(G)\ge b$.

Now, let $L'$ be a $b$-edge-list assignment of $G$. Let $g$ be an $L'$-edge-colouring such that $g(uv)$ are distinct for $u\in V(H)$; and for every $w\in V(K)$, we have $g(wv)\neq g(uv)$ for all $u\in V(H)$. All remaining edges have arbitrary colours chosen from their lists. Then $g$ is strongly rainbow connected. Hence, $src^\ell(G)\le b$.
\end{proof}

\begin{lemma}\label{preslem2}
Let $b\ge 3$. Then there exists a connected graph $G$ such that $rc^\ell(G)=2$ and $src^\ell(G)=b$.
\end{lemma}

\begin{proof}
Let $G$ be a graph consisting of two vertex-disjoint cliques $H$ and $K$, and another vertex $v$ connected to all vertices of $H$ and $K$. Clearly, $rc^\ell(G)\ge 2$. Let $L$ be a $2$-edge-list assignment of $G$. We may assume that $|L(e)|=2$ for all $e\in E(G)$. We obtain an $L$-edge-colouring $f$. Let $z\in V(H)$, and $\alpha\in L(zv)$. Let $f(uv)=\alpha$ for all $u\in V(H)$ such that $\alpha\in L(uv)$. Let $f(wv)\in L(wv)\setminus\{\alpha\}$ for all $w\in V(K)$, and let $\Phi=\{f(wv):w\in V(K)\}$ be the set of these colours. For all $u\in V(H)$ such that $\alpha\not\in L(uv)$, and $L(uv)\not\subset\Phi$ or $L(uz)\setminus\{\alpha\}\not\subset\Phi$, let $f(uv)\in L(uv)\setminus\Phi$ or $f(uz)\in L(uz)\setminus\{\Phi\cup\{\alpha\}\}$. Now, let $w\in V(H)$ be a remaining vertex, so that $\alpha\not\in L(wv)$ and $L(wv),L(wz)\setminus\{\alpha\}\subset\Phi$. We have $L(wv)=\{\beta,\gamma\}$ for some $\beta,\gamma\in\Phi$, so $\beta,\gamma\neq\alpha$. If there exists $\delta\in L(wz)$ with $\delta\in\Phi\setminus\{\beta,\gamma\}$, then $\delta\neq\alpha$, and we let $f(wv)=\beta$ and $f(wz)=\delta$. Otherwise, we may assume that $L(wz)=\{\alpha,\gamma\}$ or $\{\beta,\gamma\}$. Let $f(wv)=\beta$ and $f(wz)=\gamma$. Repeat this procedure for all such vertices $w\in V(H)$. All remaining edges have arbitrary colours chosen from their lists. Then, for all $x\in V(H)$ and $y\in V(K)$, either $xvy$ or $xzvy$ is a rainbow path. Hence, $f$ is a rainbow connected colouring, and $rc^\ell(G)\le 2$.

Setting $|V(H)|=b-1$ and $|V(K)|=(b-1)^{b-1}$, we have $src^\ell(G)=b$ by Lemma \ref{preslem1}.
\end{proof}

\begin{lemma}\label{preslem3}
Let $3\le a<b$. Then there exists a connected graph $G$ such that $rc^\ell(G)=a$ and $src^\ell(G)=b$.
\end{lemma}

\begin{proof}
We consider an example of Chen et al.~(\cite{CLLL18}, Lemma 4.10). Let $K_{m\times 1,1\times n}$ be a complete multipartite graph with $m \ge 2$ singleton classes, say $\{u_1\},\dots,\{u_m\}$. Let $U = \{u_1,\dots,u_m\}$, and $V=\{v_1,\dots, v_n\}$ be the class with $n$ vertices. Add $a\ge 3$ pendent edges at $u_1$, say $W = \{w_1, \dots , w_a\}$ is the set of pendent vertices. Let $G$ be the resulting graph, where $n=(b-1)^m+1\ge m$. 

Since $u_1$ is a universal vertex of $G$, and $w_1,\dots,w_a$ are the trivial components of $G-u_1$, we have $rc^\ell(G)=a$ by Theorem \ref{univthm}(a). Now, since $\lceil\!\sqrt[m]{n}\,\rceil = b$, by Theorems \ref{charthm}(d),  \ref{rcsrcmpthm}(b) and (\ref{ineqs}), we have  $src^\ell(G)\ge src(G)\ge src(K_{m\times 1,1\times n})=b$. Let $L$ be a $b$-edge-list assignment of $G$. We obtain an $L$-edge-colouring $c$. Let $c(u_1w_k)\in L(u_1w_k)$ be distinct for $1\le k\le a$. Then, let $c(u_iu_h)\in L(u_iu_h)\setminus\{c(u_1w_1),\dots,c(u_1w_a)\}$ for all $1\le i<h\le m$; and $c(u_1v_j)\in L(u_1v_j)\setminus\{c(u_1w_1),\dots,c(u_1w_a)\}$ for all $1\le j\le n$. Now, let $m$ be sufficiently large so that $b^{m-1}>(b-1)^m$. This inequality holds if $m>\frac{\log b}{\log b-\log (b-1)}$. Then $b^{m-1} \ge n>(b-1)^{m-1}$, and $\lceil\!\sqrt[m-1]{n}\,\rceil = b$. By Theorem \ref{rcsrclKmnthm}(b), we may choose for $c$, a strongly rainbow connected $L$-edge-colouring of the copy of $K_{m-1,n}$ with classes $U\setminus\{u_1\}$ and $V$. Now, for $x\in W$ and $y\in V(G)\setminus\{u_1\}$, $xu_1y$ is a rainbow $x-y$ geodesic. For $x,y\in V$, there exists $u\in U\setminus\{u_1\}$ such that $xuy$ is a rainbow path, and $xuy$ is a rainbow $x-y$ geodesic in $G$. Hence, $c$ is a strongly rainbow connected colouring of $G$, and $src^\ell(G) \le b$.
\end{proof}

We may now prove Theorems \ref{srcsrclcompthm} and \ref{rclsrclcompthm}.

\begin{proof}[Proofs of Theorems \ref{srcsrclcompthm} and \ref{rclsrclcompthm}]
Suppose that there exists a connected graph $G$ such that $src(G)=a$ and $src^\ell(G)=b$ (resp.~$rc^\ell(G)=a$ and $src^\ell(G)=b$). Then by (\ref{ineqs}), we have $a \le b$. If $a = 1$, then Theorem \ref{charthm}(e) gives $b = 1$. Hence, either $a = b=1$, or $2 \le a \le b$.

Conversely, given $a,b$ such that either $a = b=1$ or $2\le a\le b$, we show that there exists a connected graph $G$ with $src(G) = a$ and $src^\ell(G) = b$  (resp.~$rc^\ell(G)=a$ and $src^\ell(G)=b$). If $a=b\ge 1$, then by Theorem \ref{charthm}(g), we have $src(G)=src^\ell(G)=a$ (resp.~$rc^\ell(G)=src^\ell(G)=a$) if $G$ is the path of length $a$. Now, let $2 \le a < b$. For Theorem \ref{srcsrclcompthm}, we take the graph $G$ in Lemma \ref{preslem1} (with the same notations), where $H$ consists of a clique on $b-a+1$ vertices, and a further $a-2$ isolated vertices. We have $src^\ell(G)=b$. Now, $G-v$ has $a$ components, say $H_1,\dots,H_a$. By Theorem \ref{charthm}(c), we have $src(G)\ge a$. Let $c$ be the edge-colouring of $G$ such that $c(e)=i$ if and only if the edge $e$ has an end-vertex in $H_i$, for $1\le i\le a$. Then $c$ is strongly rainbow connected. Hence, $src(G)\le a$. For Theorem \ref{rclsrclcompthm}, the required graph exists by taking $G$ in Lemma \ref{preslem2} if $a=2$, and $G$ in Lemma \ref{preslem3} if $a\ge 3$.
\end{proof}

For the remaining inequality $rc(G)\le rc^\ell(G)$ of (\ref{ineqs}), we have not been able to prove a similar result. We propose the following problem.

\begin{prob}\label{rcrclcompprob}
Characterise all pairs of positive integers $a$ and $b$ such that, there exists a connected graph $G$ with $rc(G) = a$ and $rc^\ell(G) = b$. Is it true that $rc(G)=rc^\ell(G)$ for all connected graphs $G$?
\end{prob}

By (\ref{ineqs}) and Theorem \ref{charthm}(e), we see that in Problem \ref{rcrclcompprob}, $a$ and $b$ must necessarily satisfy $a=b=1$ or $2\le a\le b$. For any pair $a$ and $b$ such that $2\le a<b$, we have not been able to construct a connected graph $G$ such that $rc(G)=a$ and $rc^\ell(G)=b$. The question of whether or not we have $rc(G)=rc^\ell(G)$ for all connected graphs $G$ can be considered to be in a similar direction as the list colouring conjecture. We shall present some insights which may suggest that Problem \ref{rcrclcompprob} is far from simple.

Suppose that we wish to consider Problem \ref{rcrclcompprob} for $2\le a<b$ with $a$ small. Then, a necessary condition on $G$ is diam$(G)\le a$. For the case of graphs with diameter $2$, we may consider whether or not $G$ has a universal vertex. If so, then Theorem \ref{univthm} says that either $rc(G)=rc^\ell(G)$, or possibly $rc(G)=2$ and $rc^\ell(G)=3$. Otherwise, suppose that $G$ has $n$ vertices, and does not have a universal vertex. If we hope to show that $rc(G)<rc^\ell(G)$, then we may try to make $G$ to be sparse, in the hope to make $rc^\ell(G)$ large. If $G$ has the minimum possible number of edges, then a result of Erd\H{o}s and R\'enyi \cite{ER62} states that this minimum is $2n-5$. This result was  sharpened by Henning and Southey \cite{HS15}, when they characterised all graphs $G$ that attain this minimum. For a graph $G$, a \emph{degree $2$ vertex duplication} is an operation on $G$ which creates a new graph $G'$ by taking a vertex $v\in V(G)$ with $d_G(v)=2$, and connecting a new vertex $v'$ to the two neighbours of $v$. Let $G_7$ denote the graph obtained by subdividing each of three adjacent edges of $K_4$ once. Let $\mathcal G$ be the family of graphs that contain $C_5,G_7$, and the Petersen graph $\textsf P_{10}$; and $\mathcal G$ is closed under degree 2 vertex duplications. Thus, $\mathcal G$ contains the graphs as shown in Figure 3.
\indent\\[-0.2cm]
\begin{figure}[htp]
\centering
\includegraphics[width=13cm]{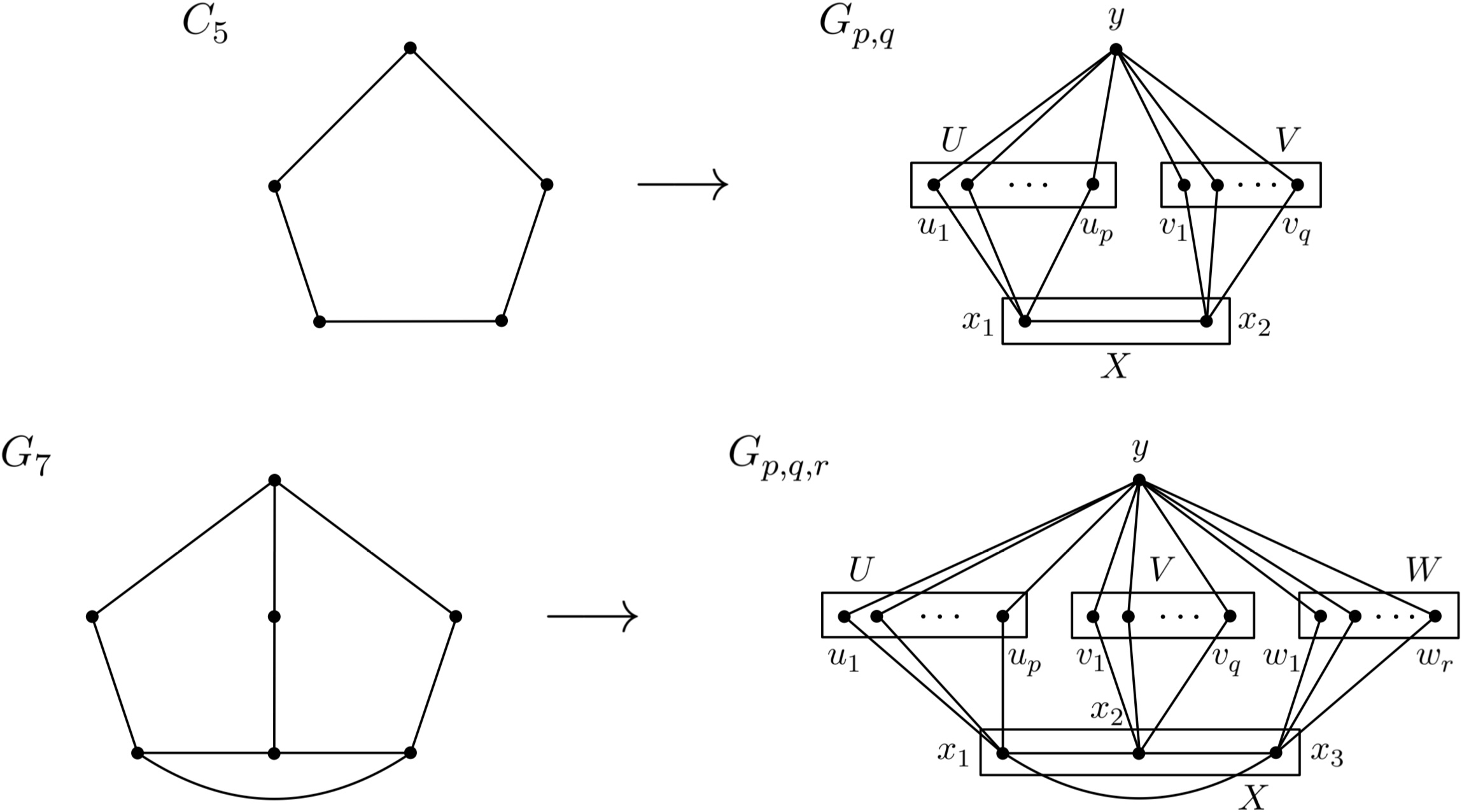}
\end{figure}
\indent\\[-0.9cm]
\begin{center}
\hspace{0.02cm}Figure 3. The graphs of $\mathcal G$
\end{center}
\indent\\[-0.7cm]

%\noindent$x_1\quad x_2\quad x_3\quad y\quad G_{p,q}$\\
%$u_1\quad u_p\quad v_1\quad v_q\quad w_1$\\
%$w_r\quad U\quad V\quad W\quad X$\\
%$\longrightarrow\quad C_5\quad G_7\quad G_{p,q,r}$

In Figure 3, for a graph which involves degree 2 vertex duplications of $C_5$, let  $u_i,v_j,x_1,x_2$, $y$ be the vertices as indicated, for $1\le i\le p$ and $1\le j\le q$. Let $U=\{u_1,\dots,u_p\}$, $V=\{v_1,\dots,v_q\}$, $X=\{x_1,x_2\}$, and we assume that $p\ge q\ge 1$. Let $t=p+q$, and $G_{p,q}$ be the resulting graph. Similarly, for a graph which involves degree 2 vertex duplications of $G_7$, let  $u_i,v_j,w_k,x_1,x_2,x_3,y$ be the vertices as indicated, for $1\le i\le p$, $1\le j\le q$ and $1\le k\le r$. Let $U=\{u_1,\dots,u_p\}$, $V=\{v_1,\dots,v_q\}$, $W=\{w_1,\dots,w_r\}$, $X=\{x_1,x_2,x_3\}$, where $p\ge q\ge r\ge 1$. Let $t=p+q+r$, and $G_{p,q,r}$ be the resulting graph. We have the following result.

\begin{thm}\label{ER-HSthm}\textup{\cite{ER62,HS15}}
Let $G$ be a graph with $n$ vertices and diameter $2$, and with no universal vertex. Then $e(G)\ge 2n-5$. Moreover, equality holds if and only if $G\in\mathcal G$.
\end{thm}

We shall prove the following result about the graphs of $\mathcal G$.

\begin{thm}\label{minthm}
For all graphs $G\in\mathcal G$, we have $rc(G)=rc^\ell(G)\in\{3,4\}$.
\end{thm}

Before we prove Theorem \ref{minthm}, we gather some auxiliary results. Firstly, the list chromatic number of $K_{t_2\times 2,2\times 3}$ was determined by Gravier and Maffray \cite{GM98}. 

\begin{thm}\label{GMthm}\textup{\cite{GM98}}
For $t_2\ge 1$, we have $\chi_\ell(K_{t_2\times 2,2\times 3})=t_2+2$. Hence, $\chi_\ell(K_{2,3,3})=3$.
\end{thm}

By Theorem \ref{KieOhbthm}, we have $\chi_\ell(K_{3,3,3})=4$. The following lemma says that $K_{3,3,3}$ is actually very nearly $3$-choosable.

%\begin{lemma}\label{K333lm}
%Let $G=K_{3,3,3}$ with classes $S_1,S_2,S_3$, and $v\in S_i$. Let $L$ be a $3$-list assignment of $G$. Then there exists an $L$-colouring $c$ of $G$ such that, $c$ is a proper colouring of $G-v$, and at most one vertex of $G-S_i$ may have the colour $c(v)$.
%\end{lemma}
%
%\begin{proof}
%Let $L$ be a $3$-list assignment of $G$. By Theorem \ref{GMthm} with $t_2=1$, there exists an $L$-colouring $c$ of $G-v$ which is proper for $G-v$. Now, if we can choose $c(v)\in L(v)\setminus\{c(u):u\in S_j\cup S_k\}$, where $\{j,k\}=[3]\setminus\{i\}$, then $c$ is a proper colouring of $G$. Otherwise, we have $L(v)\subset\{c(u):u\in S_j\cup S_k\}$, and there exist $\alpha,\beta\in L(v)$ such that $\alpha,\beta\in\{c(u):u\in S_\ell\}$ and $\alpha,\beta\not\in\{c(u):u\in S_m\}$, for some $\ell\in\{j,k\}$ and $\{m\}=\{j,k\}\setminus\{\ell\}$. The lemma follows by letting $c(v)=\alpha$ or $c(v)=\beta$, whichever of $\alpha,\beta$ occurs once in $S_\ell$.
%\end{proof}

\begin{lemma}\label{K333lm}
Let $G=K_{3,3,3}$ with classes $S_1,S_2,S_3$, and $u\in S_1$. Let $L$ be a $3$-list assignment of $G$. Then there exists an $L$-colouring $c$ of $G$ such that, $c$ is a proper colouring of $G-u$, and one of the following holds.
\begin{enumerate}
\item[(i)] $c(u)=c(x)$ for a unique $x\in S_2\cup S_3$, and $c(u),c(x)\not\in\{c(z):z\in V(G)\setminus\{u,x\}\}$. Moreover, if $w\in S_i$ for some $i\in\{2,3\}$, and $x=w$, then there exists an $L$-colouring $c'$ of $G$ such that, $c'\equiv c$ on $G-u$, $c'(u)\not\in\{c'(z):z\in (S_1\cup S_{5-i}\cup\{w\})\setminus\{u\}\}$, and $c'(u)=c'(z)$ for all $z\in S_i\setminus\{w\}$.
\item[(ii)] $c(u)\not\in\{c(z):z\in S_2\cup S_3\}$, and so $c$ is a proper colouring of $G$. Moreover, if $c(u)=c(v)$ for some $v\in S_1\setminus\{u\}$, then there exists an $L$-colouring $c''$ of $G$ such that, $c''\equiv c$ on $G-u$, and $c''(u)\not\in\{c''(z):z\in (S_1\cup S_3)\setminus\{u\}\}$.
\end{enumerate}
\end{lemma}

\begin{proof}
By Theorem \ref{GMthm}, there exists an $L$-colouring $c$ of $G-u=K_{2,3,3}$ which is proper. Suppose first that there exist $\alpha,\beta\in L(u)$ such that $\alpha,\beta\in\{c(z):z\in S_i\}$, for some $i\in\{2,3\}$. We may assume that $\alpha$ occurs exactly once in $S_i$ at some $x$, and the first part of (i) holds by letting $c(u)=\alpha$. Moreover, if $x=w\in S_i$, then $c(z)=\beta$ for all $z\in S_i\setminus\{w\}$, otherwise we could have let $c(u)=\beta$. Let the $L$-colouring $c'$ of $G$ be such that $c'\equiv c$ on $G-u$, and $c'(u)=\beta$. Then the second part of (i) holds.

Now, suppose that no such $\alpha,\beta\in L(u)$ exist. Then we can choose $c(u)\in L(u)\setminus\{c(z):z\in S_2\cup S_3\}$, and the first part of (ii) holds. Moreover, if $c(u)=c(v)$ where $v\in S_1\setminus\{u\}$, then there exist $\gamma,\delta\in L(u)\setminus\{c(u)\}$ such that $\gamma\in\{c(z):z\in S_2\}$ and $\delta\in\{c(z):z\in S_3\}$, otherwise we could have chosen another colour of $L(u)\setminus\{c(z):z\in S_2\cup S_3\}$ for $c(u)$. Let the $L$-colouring $c''$ of $G$ be such that $c''\equiv c$ on $G-u$, and $c''(u)=\gamma$. Then the second part of (ii) holds.
\end{proof}

Secondly, the determination of $rc^\ell(G_{5,2,2})$ will turn out to be the most difficult. We focus on this case in the following result.

\begin{lemma}\label{G522UB}
$rc^\ell(G_{5,2,2})\le 3$.
\end{lemma}

\begin{proof}
Let $L$ be a $3$-edge-list assignment of $G_{5,2,2}$. We construct an $L$-edge-colouring of $G_{5,2,2}$ which is rainbow connected, as follows. In summary, we identify the edges in $\{u_4,u_5\}\cup V\cup W\cup X$ with the vertices of $K_{3,3,3}$, and apply Lemma \ref{K333lm}. We repeat this with the edges incident to $y$. We then modify the two edge-colourings that are obtained, if necessary. Finally, we choose the colours of the edges $u_1x_1,u_2x_1,u_3x_1$.

Let $H=K_{3,3,3}$ with classes $S_1=\{u_4x_1,u_5x_1,x_2x_3\}$, $S_2=\{v_1x_2,v_2x_2,x_1x_3\}$, $S_3=\{w_1x_3,w_2x_3,x_1x_2\}$. We apply Lemma \ref{K333lm} to $H$ with $u=u_4x_1$, and obtain the colours of the edges of $S_1\cup S_2\cup S_3$ under an $L$-edge-colouring $c$. We may assume $c(u)\not\in\{c(z):z\in S_3\}$, otherwise we may switch the roles of $S_2$ and $S_3$. If (i) holds and $x=w=x_1x_3\in S_2$, we obtain the colouring $c'$ as described. Let $f$ be the updated colouring, so that $f\equiv c$ or $c'$ on $S_1\cup S_2\cup S_3$, whichever of $c,c'$ is defined. Let $f(x_1x_2)=\alpha_2$, $f(x_1x_3)=\alpha_3$, $f(x_2x_3)=\alpha_1$, $f(u_4x_1)=\zeta$. Then $\alpha_1,\alpha_2,\alpha_3$ are distinct, and $\alpha_2,\alpha_3,f(w_1x_3),f(w_2x_3)\neq\zeta$. By switching $v_1$ and $v_2$ if necessary, we have the following three possibilities.
\begin{enumerate}
\item[(I)] $f(v_1x_2)=\zeta$, and $\alpha_1,f(u_5x_1),f(v_2x_2)\neq\zeta$.
\item[(II)] $f(v_1x_2)=f(v_2x_2)=\zeta$, and $\alpha_1,f(u_5x_1)\neq\zeta$.
\item[(III)] $f(v_1x_2),f(v_2x_2)\neq\zeta$, and $f(u_5x_1)=\zeta$, $\alpha_1=\zeta$ are possible.
\end{enumerate}

Next, we consider the edges incident to $y$. Let $S_3'=\{u_1,u_2,u_3\}$.\\[1ex]
\emph{Case 1.} $f(v_1x_2)\neq f(v_2x_2)$.\\[1ex]
\indent This case occurs if (I) holds, may occur if (III) holds, and cannot occur if (II) holds. Let $\tilde{S}_1=\{u_4,u_5,w_1\}$ and $\tilde{S}_2=\{v_1,v_2,w_2\}$. Apply Lemma \ref{K333lm} to $\tilde{S}_1,\tilde{S}_2,S_3'$, with $u=u_4$, and:
\begin{enumerate}
\item[(1a)] $w=v_1\in \tilde{S}_2$, if (i) holds with $x\in \tilde{S}_2$, and $f(v_1x_2)=\zeta$;
\item[(1b)] $v=u_5\in \tilde{S}_1$, if (ii) holds and $f(u_5x_1)=\zeta$.
\end{enumerate}
We obtain the colours of the edges at $y$ under an $L$-edge-colouring $\tilde{c}$, where $\tilde{c}(zy)$ is the colour of $z\in\tilde{S}_1\cup \tilde{S}_2\cup S_3'$. If (1a) or (1b) holds, we obtain the colouring $\tilde{c}'$ or $\tilde{c}''$ as described in Lemma \ref{K333lm}. Let $f$ be the updated colouring, so that $f\equiv \tilde{c},\tilde{c}'$ or $\tilde{c}''$ on $\tilde{S}_1\cup \tilde{S}_2\cup S_3'$, whichever of $\tilde{c},\tilde{c}',\tilde{c}''$ is defined. \\[1ex]
\emph{Case 2.} $f(v_1x_2)= f(v_2x_2)=\beta$.\\[1ex] 
\indent This case occurs if (II) holds, may occur if (III) holds, and cannot occur if (I) holds.\\[1ex] 
\emph{Subcase 2.1.} $f(w_1x_3)\neq f(w_2x_3)$.\\[1ex] 
\indent Let $\hat{S}_1=\{u_4,u_5,v_1\}$ and $\hat{S}_2=\{v_2,w_1,w_2\}$. Apply Lemma \ref{K333lm} to $\hat{S}_1,\hat{S}_2,S_3'$, with $u=u_4$, and:
\begin{enumerate}
\item[(2a)] $w=v_2\in \hat{S}_2$, if (i) holds with $x\in \hat{S}_2$, and $f(v_1x_2)=f(v_2x_2)=\beta=\zeta$;
\item[(2b)] $v=u_5\in \hat{S}_1$, if (ii) holds and $f(u_5x_1)=\zeta$;
\item[(2c)] If (ii) holds, $v=v_1\in \hat{S}_1$ and $f(v_1x_2)=f(v_2x_2)=\beta=\zeta$, then we update the colouring $f$ on $S_1\cup S_2\cup S_3$ by reverting back to the colouring $c$. We also update $\zeta$, so that $f(u_4x_1)=f(x_1x_3)=\zeta=\alpha_3$ is different from $f(v_1x_2)=f(v_2x_2)=\beta$.
\end{enumerate}
We similarly obtain the colours of the edges at $y$ under an $L$-edge-colouring $\hat{c}$. If (2a) or (2b)  holds, we obtain the colouring $\hat{c}'$ or $\hat{c}''$ as described in Lemma \ref{K333lm}. Let $f$ be the updated colouring, so that $f\equiv \hat{c},\hat{c}'$ or $\hat{c}''$ on $\hat{S}_1\cup \hat{S}_2\cup \hat{S}_3$, whichever of $\hat{c},\hat{c}',\hat{c}''$ is defined. \\[1ex] 
\emph{Subcase 2.2.} $f(w_1x_3)= f(w_2x_3)=\gamma$.\\[1ex] 
\indent Let $\bar{S}_1=\{u_4,v_1,w_1\}$ and $\bar{S}_2=\{u_5,v_2,w_2\}$. Apply Lemma \ref{K333lm} to $\bar{S}_1,\bar{S}_2,S_3'$, with $u=u_4$, and:
\begin{enumerate}
\item[(2d)] $w=u_5\in \bar{S}_2$, if (i) holds with $x\in \bar{S}_2$, and $f(u_5x_1)=\zeta$;
\item[(2e)] $w=v_2\in \bar{S}_2$, if (i) holds with $x\in \bar{S}_2$, and $f(v_1x_2)=f(v_2x_2)=\beta=\zeta$;
\item[(2f)] If (ii) holds, $v=v_1\in \bar{S}_1$ and $f(v_1x_2)=f(v_2x_2)=\beta=\zeta$, then we will modify the colouring $f$ on $S_1\cup S_2\cup S_3$, with the method to be explained below.
\end{enumerate}
We similarly obtain the colours of the edges at $y$ under an $L$-edge-colouring $\bar{c}$. If (2d) or (2e) holds, we obtain the colouring $\bar{c}'$ as described in Lemma \ref{K333lm}. Let $f$ be the updated colouring, so that $f\equiv \bar{c}$ or $\bar{c}'$ on $\bar{S}_1\cup \bar{S}_2\cup S_3'$, whichever of $\bar{c},\bar{c}'$ is defined. 

Now, suppose that (2f) holds. We update $f$ on $S_1\cup S_2\cup S_3$, preferably to a proper colouring of $H$ whenever possible. Consider the colouring $c$ which gave the colouring $c'$. We have $c(u_4x_1)=c(x_1x_3)=\alpha_3$, and $\alpha_3,\beta\in L(u_4x_1)$. Note that $L(u_4x_1)\subset\{c(z):z\in S_2\cup S_3\}$, otherwise $c(u_4x_1)$ can be updated to give a proper colouring of $H$. Thus we may have $\alpha_2,\gamma\in L(u_4x_1)$, where $\alpha_2=\gamma$ is possible. 

If $\alpha_2\neq\gamma$, then we update $f$ by reverting $f$ back to $c$. Then, note that $u_4x_1x_2x_3w_1$ and  $u_4x_1x_2x_3w_2$ are rainbow $u_4-w_1$ and $u_4-w_2$ paths.

Now, let $\alpha_2=\gamma$. We have $L(u_4x_1)=\{\alpha_2,\alpha_3,\beta\}$. Let $f(u_5x_1)=\delta$, where $\delta\neq \alpha_2,\alpha_3,\beta$ ($\delta=\alpha_1$ is possible). If $L(v_1x_2)$ or $L(v_2x_2)$ contains a colour $\eta$ other than $\alpha_1,\alpha_2,\beta,\delta$ ($\eta=\alpha_3$ is possible), then we may reduce to (I) and thus Case 1 by updating $f(v_1x_2)=\eta$ or $f(v_2x_2)=\eta$, and $f(u_4x_1)=\beta$. 

($\ast$) Suppose first that $\delta=\alpha_1$. If $L(x_1x_3)$ contains a colour $\eta$ other than $\alpha_1,\alpha_2,\alpha_3$ ($\eta=\beta$ is possible), then we have a proper colouring of $H$ if we update $f(x_1x_3)=\eta$ and $f(u_4x_1)=\alpha_3$. Now, we have a proper colouring of $H$ by updating $f(v_1x_2)=f(v_2x_2)=f(x_1x_3)=\alpha_1$, $f(u_5x_1)\not\in L(u_5x_1)\setminus\{\alpha_1,\alpha_2\}$ and $f(x_2x_3)\not\in L(x_2x_3)\setminus\{\alpha_1,\alpha_2\}$. 

Now, suppose $\delta\neq\alpha_1$. ($\#$) If we can update $f(v_1x_2)=\delta$, then we do so, and update $f(u_4x_1)=\beta$. Note that $u_4,v_2$, and $u_5,v_1$, were already connected by the rainbow paths $u_4yv_2$ and  $u_5yv_1$ under $f$. Otherwise, we have $L(v_1x_2)=\{\alpha_1,\alpha_2,\beta\}$. We may return to the beginning of Subcase 2.2, switch the roles of $v_1$ and $v_2$, and apply the same argument, to obtain $L(v_2x_2)=\{\alpha_1,\alpha_2,\beta\}$. Now, update $f(u_4x_1)=\alpha_3$, $f(v_1x_2)=f(v_2x_2)=\alpha_1$, and $f(x_1x_3)=\eta\in L(x_1x_3)\setminus\{\alpha_2,\alpha_3\}$ ($\eta=\alpha_1,\beta$ or $\delta$ are possible). If $\eta=\alpha_1$, update $f(x_2x_3)\in L(x_2x_3)\setminus\{\alpha_1,\alpha_2\}$. If $\eta=\delta$, update $f(u_5x_1)\in L(u_5x_1)\setminus\{\alpha_2,\delta\}$, and note that $f(u_5x_1)\neq\alpha_1$, otherwise we may apply the previous argument ($\ast$). We see that $f$ becomes a proper colouring of $H$.\\[1ex]
\indent Now, for both Cases 1 and 2, let $f$ be the updated $L$-edge-colouring of $G_{5,2,2}-\{u_1x_1,u_2x_1$, $u_3x_1\}$, and update $f(u_4x_1)=\zeta$. When we applied Lemma \ref{K333lm} to the edges incident to $y$ with $u=u_4$, if (i) holds with $w\in S_3'$, we may assume that $w=u_1$. If we do not have $f(u_1y)=f(u_2y)=f(u_3y)$, we assume that $f(u_1y)\neq f(u_2y)$. Then, for some $\ell\in\{2,3\}$, one of the following holds.
\begin{enumerate}
\item[(A)] $f(u_1y)\neq\alpha_\ell$ and $f(u_2y)\neq\alpha_{5-\ell}$.
\item[(B)] $f(u_1y)=f(u_2y)=f(u_3y)=\alpha_{5-\ell}$ and $f(u_4y),f(u_5y)\neq\alpha_{5-\ell}$.
\end{enumerate}
Let $f(u_1x_1)\in L(u_1x_1)\setminus\{\alpha_\ell,\zeta\}$ if $f(u_1y)=f(u_4y)$, and $f(u_1x_1)\in L(u_1x_1)\setminus\{f(u_1y),\alpha_\ell\}$ if $f(u_1y)\neq f(u_4y)$. Let $f(u_2x_1)\in L(u_2x_1)\setminus\{f(u_2y),\alpha_{5-\ell}\}$ if $f(u_1y)\neq f(u_2y)$, and $f(u_2x_1)\in L(u_2x_1)\setminus\{f(u_1x_1),\alpha_{5-\ell}\}$ if $f(u_1y)=f(u_2y)=f(u_3y)$. Let $f(u_3x_1)\in L(u_3x_1)\setminus\{f(u_1x_1),f(u_2x_1)\}$. 

For Case 1, we have the following rainbow $y-x_\ell$ and $y-x_{5-\ell}$ paths: If (A) holds, then we take $yu_1x_1x_\ell$ or $yu_4x_1x_\ell$; and $yu_2x_1x_{5-\ell}$ if $f(u_1y)\neq f(u_2y)$; and $yu_ix_1x_2$ and $yu_hx_1x_3$ for some $i,h\in\{1,2,3\}$ if $f(u_1y)=f(u_2y)=f(u_3y)$. If (B) holds and $f(u_5x_1)=\zeta$, then we take $yu_1x_1x_\ell$; and $yu_4x_1x_{5-\ell}$ or $yu_5x_1x_{5-\ell}$. Now, suppose that (B) holds and $f(u_5x_1)\neq\zeta$. If $f(v_1x_2)\neq\zeta$, we update $f(u_4y)$ so that $f(u_4y)\in L(u_4y)\setminus\{\alpha_{5-\ell},\zeta\}$. We have $yu_1x_1x_\ell$ is a rainbow $y-x_\ell$ path. We also have a rainbow $y-x_{5-\ell}$ path as follows: If $f(v_1x_2)\neq\zeta$, we take $yu_4x_1x_{5-\ell}$. Consider $f(v_1x_2)=\zeta$. If $\ell=3$, we take $yv_1x_2$ if $f(v_1y)\neq \zeta$; and $yu_4x_1x_2$ if $f(v_1y)=\zeta$. If $\ell=2$, we take $yu_4x_1x_3$ if $f(u_4y)\neq\zeta$; and $yv_1x_2x_3$ or $yv_1x_2x_1x_3$ if $f(u_4y)=\zeta$. 

For Case 2, if we have the case ($\#$) of (2f), recall that $f(u_5x_1)=f(v_1x_2)$ and $f(w_1x_3)=f(w_2x_3)$. Then $yu_1x_1x_2$, $yu_5x_1x_2$ or $yv_1x_2$; and $yw_1x_3$ or $yw_2x_3$, are rainbow $y-x_2$ and $y-x_3$ paths. Otherwise, $f(v_1x_2)=f(v_2x_2)$, and $yv_1x_2x_3$, $yv_2x_2x_3$, $yv_1x_2x_1x_3$ or $yv_2x_2x_1x_3$ contains rainbow $y-x_2$ and $y-x_3$ paths. 

($\dag$) For both Cases 1 and 2, it is easy to verify that any other two vertices of $G_{5,2,2}$ are connected by a rainbow path under $f$. For $y$ and $x_1$, we have $yu_ix_1$ is rainbow $y-x_1$ path for some $u_i\in U$. If $z_1,z_2\in U\cup V\cup W$, then there is a rainbow $z_1-z_2$ path, using either $y$ or some of $x_1,x_2,x_3$. For $x_2$ and $u_i\in U$, either $x_2x_1u_i$ or $x_2x_3x_1u_i$ is a rainbow $x_2-u_i$ path. A similar argument applies for any $z_1\in X$ and $z_2\in U\cup V\cup W$ which are non-adjacent. 

We have $f$ is a rainbow connected colouring of $G_{5,2,2}$, and hence $rc^\ell(G_{5,2,2})\le 3$.
\end{proof}

We are now ready to prove Theorem \ref{minthm}.

\begin{proof}[Proof of Theorem \ref{minthm}]
By Theorems \ref{CJMZPetthm} and \ref{Petthm}, we have $rc(\textsf P_{10})=rc^\ell(\textsf P_{10})=3$. It remains to consider the graphs of Figure 3. We shall prove the following.
\begin{equation}\label{minthmeq1}
rc(G_{p,q})=rc^\ell(G_{p,q})=
\left\{
\begin{array}{ll}
3, & \textup{if $2\le t\le 7$ or $(p,q)=(6,2),(5,3),(4,4),(6,3)$},\\[0.2ex]
4, & \textup{if $t\ge 10$ or $(p,q)=(7,1),(8,1),(7,2),(5,4)$}.
\end{array}
\right.
\end{equation}
\indent\\[-0.7cm]
\begin{equation}\label{minthmeq2}
rc(G_{p,q,r})=rc^\ell(G_{p,q,r})=
\left\{
\begin{array}{ll}
3, & \textup{if $3\le t\le 8$ or}\\
& \quad\textup{$(p,q,r)=(7,1,1),(4,4,1),(5,2,2),(3,3,3)$},\\[0.2ex]
4, & \textup{if $t\ge 10$ or}\\
& \quad\textup{$(p,q,r)=(6,2,1),(5,3,1),(4,3,2)$}.
\end{array}
\right.
\end{equation}

We can relate $G_{p,q}$ or $G_{p,q,r}$ to $K_{2,t}$, by contracting the set $X$ to a single vertex. If $G_{p,q}$ or $G_{p,q,r}$ is given an edge-colouring $c$, then we associate the vertices $u_i,v_j,w_k$ with the $2$-vectors $\vec{u}_i,\vec{v}_j,\vec{w}_k$, where $\vec{u}_{i1} = c(u_iy)$, $\vec{u}_{i2} = c(u_ix_1)$, $\vec{v}_{j1} = c(v_jy)$, $\vec{v}_{j2} = c(v_jx_2)$, $\vec{w}_{k1} = c(w_ky)$, and $\vec{w}_{k2} = c(w_kx_3)$.

Since $G_{1,1}=C_5$, we have $rc(G_{1,1})=rc^\ell(G_{1,1})=3$ by Theorems \ref{rccyclethm} and \ref{rclcyclethm}. Now for 
$G_{p,q}$, let $t\ge 3$, so that $p\ge 2$. Suppose there exists a rainbow connected $2$-edge-colouring $c$ of $G_{p,q}$. For $i=1,2$, the unique $u_i-v_1$ path of length $2$ is $u_iyv_1$, and so must be rainbow. Thus, $c(u_1y)=c(u_2y)$. Similarly, considering $x_2,x_1$ in place of $v_1,y$, we have $c(u_1x_1)=c(u_2x_1)$. But then, there is no rainbow $u_1-u_2$ path, a contradiction. Hence, $rc(G_{p,q})\ge 3$. For $G_{p,q,r}$, if we have a $2$-edge-colouring of $G_{p,q,r}$, then two of $u_1y,v_1y,w_1y$ have the same colour, and so two of $u_1,v_1,w_1$ cannot have a rainbow path connecting them. Hence, $rc(G_{p,q,r})\ge 3$. 

Now for $G_{p,q}$, let $t\ge 10$ or $(p,q)=(7,1),(8,1), (7,2), (5,4)$. Suppose that there exists a rainbow connected $3$-edge-colouring $c'$ of $G_{p,q}$, using colours $1,2,3$. If there exist $z_1,z_2\in U\cup V$ such that $\vec{z}_1=\vec{z}_2$, then there cannot exist a rainbow $z_1-z_2$ path. This happens if $t\ge 10$. Otherwise, if $(p,q)=(7,1),(8,1), (7,2), (5,4)$, then the associated vectors of $U\cup V$ must be distinct. For $(p,q)\neq (7,1)$, we have $c'(z_1y)=c'(z_2y)=c'(z_3y)$, for some $z_1,z_2,z_3$ such that $z_1,z_2$ are in one of $U,V$, and $z_3$ is in the other, and there does not exist a rainbow path from $z_3$ to $z_1$ or $z_2$. For $(p,q)=(7,1)$, we may assume $c'(u_iy)=c'(u_\ell x_1)=1$ for $i=1,2,3$ and $\ell=1,4$; $c'(u_iy)=c'(u_\ell x_1)=2$ for $i=4,5,6$ and $\ell=2,5$; and $c'(u_7y)=c'(v_1y)=c'(u_\ell x_1)=3$ for $\ell=3,6$. If $c'(x_1x_2)=\alpha$, then there does not exist a rainbow path from $x_2$ to $u_\alpha$ or $u_{\alpha+3}$. We always have a contradiction, and hence $rc(G_{p,q})\ge 4$. We may argue similarly for $G_{p,q,r}$, where $t\ge 10$ or $(p,q,r)=(6,2,1),(5,3,1),(4,3,2)$. We have $U\cup V\cup W$ in place of $U\cup V$, and for the vertices $z_1,z_2,z_3$, we have $\{z_1,z_2\}$ and $z_3$ are in two of $U,V,W$. We may obtain  $rc(G_{p,q,r})\ge 4$.

Next, let $L$ be a $4$-edge-list assignment of $G_{p,q}$ or $G_{p,q,r}$. We construct an $L$-edge-colouring $f$. For $G_{p,q}$, let $f(e)\in L(e)$ be distinct for $e\in\{u_1y,u_1x_1,v_1y,v_1x_2,x_1x_2\}$, except possibly $f(u_1x_1)=f(v_1x_2)$. Let $f(u_iy)\in L(u_iy)$ and $f(u_ix_1)\in L(u_ix_1)$ such that $f(u_1y),f(u_1x_1),f(u_iy),f(u_ix_1)$ are distinct for $2\le i\le p$. Let $f(v_jy)\in L(v_jy)$ and $f(v_jx_2)\in L(v_jx_2)$ such that $f(v_1y),f(v_1x_2),f(v_jy),f(v_jx_2)$ are distinct for $2\le j\le q$. For $G_{p,q,r}$, consider $H=K_{2,2,2,3}$ with classes $S_1=\{u_1y,x_2x_3\}$, $S_2=\{v_1y,x_1x_3\}$, $S_3=\{w_1y,x_1x_2\}$, $S_4=\{u_1x_1,v_1x_2,w_1x_3\}$, so that $L$ gives a $4$-list assignment of $H$. Since $\chi_\ell(H)=4$ by Theorem \ref{GMthm}, there exists a proper $L$-colouring of $H$. Let $f$ be the resulting $L$-edge-colouring of the edges of $S_1\cup S_2\cup S_3\cup S_4$. Choose the colours $f(u_iy),f(u_ix_1),f(v_jy),f(v_jx_2)$ as before, and let $f(w_ky)\in L(w_ky)$ and $f(w_kx_3)\in L(w_kx_3)$ such that $f(w_1y),f(w_1x_3),f(w_ky),f(w_kx_3)$ are distinct for $2\le k\le r$.
%
%First, choose distinct $f(u_1y)\in L(u_1y)$, $f(u_1x_1)\in L(u_1x_1)$ and $f(x_1x_2)\in L(x_1x_2)$. Then, choose $f(v_1y)\in L(v_1y)\setminus\{f(u_1y)\}$ and $f(v_1x_2)\in L(v_1x_2)\setminus\{f(v_1y),f(x_1x_2)\}$. Then, choose $f(u_iy)\in L(u_iy)$ and $f(u_ix_1)\in L(u_ix_1)$ such that $f(u_1y)$, $f(u_1x_1)$, $f(u_iy)$, $f(u_ix_1)$ are distinct, for every $2\le i\le p$. Choose $f(v_jy)\in L(v_jy)$ and $f(v_jx_2)\in L(v_jx_2)$ such that $c(v_1y)$, $c(v_1x_2)$, $c(v_jy)$, $c(v_jx_2)$ are distinct, for every $2\le j\le q$. For the case of $G_{p,q,r}$, in addition, we choose in the following order:
%\begin{itemize}
%\item $f(x_2x_3)\in L(x_2x_3)\setminus\{f(v_1y),f(v_1x_2),f(x_1x_2)\}$; 
%\item $f(w_1y)\in L(w_1y)\setminus\{f(u_1y),f(v_1y),f(x_2x_3)\}$; 
%\item $f(x_1x_3)\in L(x_1x_3)\setminus\{f(u_1y),f(u_1x_1),f(w_1y)\}$; 
%\item $f(w_1x_3)\in L(w_1x_3)\setminus\{f(w_1y),f(x_1x_3),f(x_2x_3)\}$;
%\item $f(w_ky)\in L(w_ky)$ and $f(w_kx_3)\in L(w_kx_3)$ such that $c(w_1y)$, $c(w_1x_3)$, $c(w_ky)$, $c(w_kx_3)$ are distinct, for every $2\le k\le r$.
%\end{itemize}

We show that $f$ is a rainbow connected colouring of $G_{p,q}$ or $G_{p,q,r}$. For $G_{p,q}$, we have $x_1u_1y$, $x_2v_1y$ are rainbow $x_1-y$, $x_2-y$ paths. For $u_i,u_\ell\in U$, either $u_iyu_\ell$ is a rainbow $u_i-u_\ell$ path; or if $f(u_iy)=f(u_\ell y)$, then $u_ix_1u_1yu_\ell$ is rainbow, since $f(u_1y)\neq f(u_hy)$ for all $2\le h\le p$, so that $u_i,u_\ell\neq u_1$. Similar arguments hold for $v_j,v_\ell\in V$, and for $u_i\in U$, $v_j\in V$. Note that for the latter, if $f(u_iy)=f(v_jy)$, then $u_i\neq u_1$ or $v_j\neq v_1$, so $u_ix_1u_1yv_j$ or $u_iyv_1x_2v_j$ is a rainbow $u_i-v_j$ path. For $x_1$ and $v_j\in V$, either $x_1x_2v_j$ is a rainbow $x_1-v_j$ path; or if $f(x_1x_2)=f(v_jx_2)$, then $x_1x_2v_1yv_j$ is rainbow, since $f(x_1x_2)\neq f(v_1x_2)$, so that $v_j\neq v_1$. A similar argument holds for $x_2$ and $u_i\in U$. For $G_{p,q,r}$, analogous arguments also hold. Hence, $f$ is a rainbow connected colouring, and $rc^\ell(G_{p,q})\le 4$ and $rc^\ell(G_{p,q,r})\le 4$.

Now, let $L'$ be a $3$-edge-list assignment of $G_{p,q}$, where $3\le t\le 7$ or $(p,q) = (6,2),(5,3)$, $(4,4),(6,3)$. We construct an $L'$-edge-colouring $g$. Partition $U\cup V=S_1\cup S_2\cup S_3$ such that for $h=1,2,3$, we have $|S_h|=s_h\in\{1,2,3\}$, as follows. For $(p,q)\neq (6,1),(4,4)$, let $S_1=V$, $S_2=\{u_1,\dots,u_{\lceil p/2\rceil}\}$, $S_3=U\setminus S_2$. For $(p,q)=(6,1)$, let $S_1=\{u_6,v_1\}$, $S_2=\{u_1,u_2,u_3\}$, $S_3=\{u_4,u_5\}$. For $(p,q)=(4,4)$, let $S_1=\{u_1,u_2,u_3\}$, $S_2=\{v_1,v_2,v_3\}$, $S_3=\{u_4,v_4\}$. Then $L'$ gives a $3$-list assignment of $H=K_{s_1,s_2,s_3}$ with classes $S_1,S_2,S_3$, where $z\in S_h$ has the list $L'(zy)$. We apply Theorem \ref{GMthm}, or Lemma \ref{K333lm} for the case $(p,q)=(6,3)$, to obtain a colouring of $H$. This gives the colours $g(zy)\in L'(zy)$ for $z\in U\cup V$ such that $z_1yz_2$ is a rainbow $z_1-z_2$ path whenever $z_1,z_2$ are in different sets from $S_1,S_2,S_3$, except possibly when $(p,q)=(6,3)$ and say $z_1=v_1$, $z_2=u_1$. For $S_h=\{u_i,u_{i+1},\dots\}\subset U$ for some $i$, by the operation ($\ddag$), we mean we choose $g(zx_1)\in L'(zx_1)$ to be distinct over $z\in S_h$, such that $g(u_ix_1)\neq g(u_iy)$. We have the operation ($\ddag'$) if in addition, $s_h\ge 2$ and $g(u_{i+1}x_1)\neq g(u_{i+1}y)$. These two operations are defined analogously for $S_h\subset V$, and also for $S_h\subset W$ when we consider the graphs $G_{p,q,r}$ later.
%
%We define the following method of choosing  the colours of the edges between $S_h\subset U$ and $x_1$, and between $S_h\subset V$ and $x_2$.
%\begin{enumerate}
%\item[($\ddag$)] If $S_h=\{u_i,u_{i+1},\dots\}\subset U$ for some $i$, let $g(u_\ell x_1)\in L'(u_\ell x_1)$ be distinct over $u_\ell\in S_h$ such that $g(u_1x_1)\neq g(u_1y)$, and $g(u_2x_1)\neq g(u_2y)$ if $s_h\ge 2$. For $S_h=\{v_1,v_2,\dots\}\subset V$, we may similarly choose distinct $g(v_\ell x_2)\in L'(v_\ell x_2)$ over $v_\ell\in S_h$.
%\end{enumerate}
%At this point, we omit the case $(p,q)=(6,1)$ ($\ast$). 
%We now split into the following cases.
\begin{itemize}
\item For $q=1$, let $g(v_1x_2)\in L'(v_1x_2)\setminus\{g(v_1y),g(u_{\lceil p/2\rceil}y)\}$ and $g(x_1x_2)\in L'(x_1x_2)\setminus\{g(v_1x_2)\}$. Apply ($\ddag$) to $S_2,S_3$ so that $g(u_ix_1)\neq g(x_1x_2)$ for $i\neq \lceil\frac{p}{2}\rceil$. For $(p,q)=(6,1)$, let $g(u_6x_1)\in L'(u_6x_1)\setminus\{g(v_1x_2),g(x_1x_2)\}$.
%
%\item For $(p,q)=(6,1)$, let $g(v_1x_2)\in L'(v_1x_2)\setminus\{g(v_1y),g(u_3y)\}$ and $g(x_1x_2)\in L'(x_1x_2)\setminus\{g(v_1x_2)\}$. Apply ($\ddag$) to $S_1,S_2$ such that $g(u_ix_1)\neq g(x_1x_2)$ for $i=1,2,4,5$. Let $g(u_6x_1)\in L'(u_6x_1)\setminus\{g(v_1x_2),g(x_1x_2)\}$.
\item For $(p,q)=(4,4)$, we apply ($\ddag'$) to $S_1,S_2$. Let $g(e)\in L'(e)$ be distinct for $e\in\{u_4x_1,v_4x_2,x_1x_2\}$
\item For all other cases, we have $S_1=V$ and $S_2\cup S_3=U$. Apply ($\ddag'$) to $S_1,S_2,S_3$. For $(p,q)=(6,3)$, let $g(u_1x_1)\neq g(v_1x_2)$ and $g(x_1x_2)\in L'(x_1x_2)\setminus\{g(u_1x_1),g(v_1x_2)\}$. For $(p,q)=(6,2)$, let $g(x_1x_2)\in L'(x_1x_2)$ arbitrarily. Otherwise, we have $p\le 5$ and $s_3\le 2$. Let $g(x_1x_2)\in L'(x_1x_2)\setminus\{g(zx_1):z\in S_3\}$.
\end{itemize}
In each case, it is not hard to verify that $g$ is a rainbow connected colouring. For $z_1,z_2\in S_h$, or $z_1=v_1$, $z_2=u_1$ when $(p,q)=(6,3)$, we have a rainbow $z_1-z_2$ path by using at least one of $x_1,x_2$. Also, $yu_1x_1$ and $yv_1x_2$ are rainbow $y-x_1$ and $y-x_2$ paths. For $x_1$ and $v_j\in V$, either $x_1x_2v_j$ or $x_1u_iyv_j$ is rainbow, for some $u_i\in U$. A similar argument holds for $x_2$ and $u_i\in U$. Hence, $g$ is rainbow connected, and $rc^\ell(G_{p,q})\le 3$.

Finally, let $L'$ be a $3$-edge-list assignment of $G_{p,q,r}$, where $3\le t\le 8$ or $(p,q,r) = (7,1,1)$, $(4,4,1),(3,3,3)$. We construct an $L'$-edge-colouring $g$. Let $U\cup V\cup W=S_1\cup S_2\cup S_3$ be a partition where $|S_h|=s_h\in\{1,2,3\}$ for $h=1,2,3$, to be chosen later. Apply Theorem \ref{GMthm}, or Lemma \ref{K333lm} for the case $t=9$, to $H=K_{s_1,s_2,s_3}$ as before. We obtain the colours of the edges at $y$ under an $L'$-edge-colouring $g$. We have $z_1yz_2$ is a rainbow $z_1-z_2$ path whenever $z_1,z_2$ are in different sets from $S_1,S_2,S_3$, with possible exception for one pair $z_1,z_2$ if $t=9$.

For $p\le 3$, let $S_1=U$, $S_2=V$, $S_3=W$. When $(p,q,r)=(3,3,3)$, we may possibly have, say $g(u_1y)=g(v_1y)$. Apply ($\ddag$) to each $S_h$. Let $g(e)\in L'(e)$ be distinct for $e\in\{x_1x_2,x_1x_3,x_2x_3\}$, where for $(p,q,r)=(3,3,3)$, we let $g(u_1x_1),g(v_1x_2),g(x_1x_2)$ be distinct.

Now, let $p\ge 4$. For $(p,q,r)\neq (5,2,1),(4,2,2)$, consider $H'=K_{2,2,2}$ with classes $S_1'=\{u_px_1,x_2x_3\}$, $S_2'=\{v_qx_2,x_1x_3\}$, $S_3'=\{w_rx_3,x_1x_2\}$. Since $\chi_\ell(H')=3$ by Theorem \ref{ERTthm}(b), we may obtain colours for the edges of $S_1'\cup S_2'\cup S_3'$ under an $L'$-edge-colouring $g$. Let $g(x_1x_2)=\alpha_2$, $g(x_1x_3)=\alpha_3$, $g(x_2x_3)=\alpha_1$, $g(u_px_1)=\beta_1$, $g(v_qx_2)=\beta_2$, $g(w_rx_3)=\beta_3$.

For $(p,1,1)$ where $4\le p\le 7$, let $S_1=\{u_p,v_1,w_1\}$, $S_2=\{u_1,\dots,u_m\}$, $S_3=U\setminus(S_1\cup\{u_p\})$, where $m=\lceil\frac{p-1}{2}\rceil$. When $(p,q,r)=(7,1,1)$, we may possibly have, say $g(u_2y)=g(u_7y)$. For some $\ell\in\{2,3\}$, we have $g(u_1y)\neq\alpha_\ell$ and $g(u_{m+1}y)\neq \alpha_{5-\ell}$. Apply ($\ddag$) on $S_2,S_3$ with $g(u_1x_1)\neq \alpha_\ell$, $g(u_{m+1}x_1)\neq \alpha_{5-\ell}$; and in addition $g(u_2x_1)\neq g(u_7x_1)$ for $(p,q,r)=(7,1,1)$. Then, $yu_1x_1x_\ell$ and $yu_{m+1}x_1x_{5-\ell}$ are rainbow $y-x_\ell$ and $y-x_{5-\ell}$ paths.

For $(4,q,1)$ where $2\le q\le 4$, let $S_1=\{u_4,v_q,w_1\}$, $S_2=\{u_1,u_2,u_3\}$, $S_3=\{v_1,\dots,v_{q-1}\}$. When $(p,q,r)=(4,4,1)$, we may possibly have, say $g(u_2y)=g(w_1y)$. Consider first when we have a proper colouring of $H$. Apply ($\ddag$) on $S_2,S_3$ with $g(u_1x_1)\neq\alpha_3$, $g(v_1x_2)\neq \alpha_1$. If either $g(u_1y)\neq\alpha_3$ or $g(v_1y)\neq\alpha_1$, then $yu_1x_1x_3$ or $yv_1x_2x_3$ is a rainbow $y-x_3$ path. Otherwise, $g(u_1y)=\alpha_3$ and $g(v_1y)=\alpha_1$. If none of $yu_4x_1x_3$, $yv_qx_2x_3$, $yw_1x_3$ is rainbow, then $g(u_4y)=\beta_1$, $g(v_qy)=\beta_2$, $g(w_1y)=\beta_3$, which are distinct. Now, $z_1yz_2$ is a rainbow $z_1-z_2$ path for $z_1,z_2\in S_1$. Update $g(u_4x_1)\in L'(u_4x_1)\setminus\{\alpha_3,\beta_1\}$, so that $yu_4x_1x_3$ is rainbow. Secondly, if $(p,q,r)=(4,4,1)$ and $g(u_2y)=g(w_1y)$, apply ($\ddag$) to $S_2,S_3$ with $g(u_2x_1)\neq\alpha_3,\beta_3$. Then $yw_1x_3$ or $yu_2x_1x_3$ is a rainbow $y-x_3$ path.

Now, let $(p,q,r)=(5,2,1),(4,2,2)$. Consider $H''=K_{2,3,2}$ with classes $S_1''=\{u_px_1,x_2x_3\}$, $S_2''=\{v_1x_2,v_2x_2,x_1x_3\}$, $S_3''=\{w_1x_3,x_1x_2\}$. Since $\chi_\ell(H'')=3$ by Theorem \ref{GMthm}, we may obtain colours for the edges of $S_1''\cup S_2''\cup S_3''$ under an $L'$-edge-colouring $g$. Let $g(x_1x_2)=\alpha_2$, $g(x_1x_3)=\alpha_3$, $g(x_2x_3)=\alpha_1$, $g(u_px_1)=\beta_1$, $g(v_1x_2)=\beta_2$, $g(v_2x_2)=\beta_2'$, $g(w_1x_3)=\beta_3$. Let $S_1=\{u_1,u_2,u_3\}$.

Suppose first that $\beta_2=\beta_2'$. Let $S_2 = \{u_4, v_1\}$, $S_3 = \{u_5, v_2, w_1\}$ if $(p,q,r)=(5,2,1)$; and $S_2=\{u_4,v_1,w_1\}$, $S_3=\{v_2,w_2\}$ if $(p,q,r)=(4,2,2)$. Apply ($\ddag$) to $S_1$ with $g(u_1x_1)\neq\alpha_3$. Let $g(u_4x_1)\in L'(u_4x_1)\setminus\{\alpha_2,\beta_2\}$ if $(p,q,r)=(5,2,1)$; and $g(w_2x_3)\in L'(w_2x_3)\setminus\{\alpha_1,\beta_2\}$ if $(p,q,r)=(4,2,2)$. Then $yv_1x_2$ or $yv_2x_2$ is a rainbow $y-x_2$ path. Also, one of $yu_1x_1x_3$, $yv_1x_2x_3$, $yv_2x_2x_3$ is a rainbow $y-x_3$ path, unless $g(u_1y)=\alpha_3$ and $\{g(v_1y),g(v_2y)\}=\{\alpha_1,\beta_2\}$, in which case we take $yv_1x_2x_1x_3$ or $yv_2x_2x_1x_3$.

Now, let $\beta_2\neq\beta_2'$. For $(p,q,r)=(5,2,1)$, let $S_2=\{u_4,u_5\}$, $S_3=\{v_1,v_2,w_1\}$. We have $g(u_1y)\neq\alpha_\ell$ and $g(u_4y)\neq\alpha_{5-\ell}$ for some $\ell\in\{2,3\}$. Apply ($\ddag$) on $S_1$ with $g(u_1x_1)\neq\alpha_\ell$. Let $g(u_4x_1)\in L'(u_4x_1)\setminus\{\alpha_{5-\ell},\beta_1\}$ if $g(u_4y)=g(u_5y)$, and $g(u_4x_1)\in L'(u_4x_1)\setminus\{\alpha_{5-\ell},g(u_4y)\}$ if $g(u_4y)\neq g(u_5y)$. Then $yu_1x_1x_\ell$, and $yu_4x_1x_{5-\ell}$ or $yu_5x_1x_{5-\ell}$, are rainbow $y-x_\ell$ and $y-x_{5-\ell}$ paths. For $(p,q,r)=(4,2,2)$, let $S_2=\{u_4,v_1,v_2\}$, $S_3=\{w_1,w_2\}$. Apply ($\ddag$) to $S_1$. Let $g(w_2x_3)\in L'(w_2x_3)\setminus\{\beta_3,g(w_2y)\}$, so that $yw_2x_3$ is a rainbow $y-x_3$ path. We have $yv_1x_2$ or $yv_2x_2$ is a rainbow $y-x_2$ path, unless $g(v_1y)=\beta_2$ and $g(v_2y)=\beta_2'$, in which case, we may assume $g(u_4y)\neq \beta_2$, update $g(v_1x_2)\in L'(v_1x_2)\setminus\{\beta_2\}$, and take $yv_1x_2$.

In every case, we can easily check that any remaining pair of vertices of $G_{p,q,r}$ are connected by a rainbow path under $g$, just as in ($\dag$) in Lemma \ref{G522UB}. Hence, $rc^\ell(G_{p,q,r})\le 3$.

Finally, since $rc(G)\le rc^\ell(G)$ for any connected graph $G$ by (\ref{ineqs}), and $rc^\ell(G_{5,2,2})\le 3$ by Lemma \ref{G522UB}, the results (\ref{minthmeq1}) and (\ref{minthmeq2}) follow. This completes the proof of Theorem \ref{minthm}.
\end{proof}

%Theorem \ref{univthm} tells us that, if a graph $G$ has a universal vertex $v$, then the only possibility that $rc(G)<rc^\ell(G)$ holds is $rc(G)=2$ and $rc^\ell(G)=3$, when $G-v$ has at most two components. Theorems \ref{ER-HSthm} and \ref{minthm} imply that if $G$ is such that diam$(G)=2$, and $G$ has no universal vertex and has the minimum possible number of edges, then $rc(G)=rc^\ell(G)\in\{3,4\}$. Thus $rc^\ell(G)-rc(G)\le 1$ for all of these graphs, yet they seem to be plausible candidates when we wish to construct a graph $G$ such that $rc(G)<rc^\ell(G)$ with $rc(G)$ small. Perhaps it is true that $rc(G)=rc^\ell(G)$ for all connected graphs $G$, or that $rc(G)$ and $rc^\ell(G)$ cannot be far apart.

The graphs of the family $\mathcal G$ seem to be plausible candidates to consider, when we wish to find a graph $G$ such that $rc(G)<rc^\ell(G)$ with $rc(G)$ small. But we have shown in Theorem \ref{minthm} that $rc(G)=rc^\ell(G)$ for all $G\in\mathcal G$. 
Perhaps it is true that $rc(G)=rc^\ell(G)$ for all connected graphs $G$, or that $rc(G)$ and $rc^\ell(G)$ cannot be far apart.

\section{Conclusion}

In this paper, we have introduced and studied the list rainbow connection number $rc^\ell(G)$ and the list strong rainbow connection number $src^\ell(G)$ of a connected graph $G$. These two parameters are significantly more difficult to study than the rainbow connection number $rc(G)$ and strong rainbow connection number $src(G)$. We proposed the problem of whether or not we always have $rc(G)=rc^\ell(G)$ for all connected graphs $G$. It is hoped that the parameters $rc^\ell(G)$ and $src^\ell(G)$ can be further studied in the future, which will lead to new contributions to the topic of list colourings of graphs.

\section*{Acknowledgements}
Henry\,\, Liu\,\, is\,\, partially\,\, supported\,\, by\,\, National\,\, Natural\,\, Science\,\, Foundation\,\, of\,\, China\,\, (No.~11931002).\, Rongxia\,\, Tang,\,\, Henry\,\, Liu,\,\, Yueping\,\, Shi,\,\, and\,\, Chenming\,\, Wang\,\, are partially\, supported\, by\, National\, Key\, Research\, and\, Development\, Program\, of\, China\, (No. 2020YFA0712500).

\end{document}